\documentclass[12pt,reqno]{amsart}

\usepackage{amsmath,amssymb,mathrsfs,amsfonts,amsthm,mathptmx}
\usepackage{graphicx,cite,times,color,bm,dcolumn}
\usepackage{mathrsfs}
\usepackage{cases}

\usepackage[top=1.2in, bottom=1.2in, left=1in, right=1in]{geometry}

\numberwithin{equation}{section}
\numberwithin{figure}{section}
\numberwithin{table}{section}
\newtheorem{theorem}{Theorem}[section]
\newtheorem{lemma}{Lemma}[section]

\newtheorem{corollary}{Corollary}[section]
\newtheorem{proposition}{Proposition}[section]
\newtheorem{remark}{Remark}[section]
\newtheorem{assumption}{Assumption}[section]

\allowdisplaybreaks \allowdisplaybreaks[4]%change lines

\begin{document}
	\title[Semi-discretization for  stochastic nonlinear Maxwell equations]{Mean-square convergence of a semi-discrete scheme for  stochastic nonlinear Maxwell equations}
		
	\author{Chuchu Chen}
	\address{LSEC, ICMSEC,  Academy of Mathematics and Systems Science, Chinese Academy of Sciences, Beijing 100190, China}
	\email{chenchuchu@lsec.cc.ac.cn}
	
		\author{Jialin Hong}
	\address{LSEC, ICMSEC,  Academy of Mathematics and Systems Science, Chinese Academy of Sciences, Beijing 100190, China,
\and	
School of Mathematical Sciences, University of Chinese Academy of Sciences, Beijing 100049, China}
	\email{hjl@lsec.cc.ac.cn}

\author{Lihai Ji}
\address{Institute of Applied Physics and Computational Mathematics, Beijing 100094, China}
\email{jilihai@lsec.cc.ac.cn (Corresponding author)}

	\thanks{
		The research of C. Chen and J. Hong were supported by the NNSFC (NOs. 91130003, 11021101, 11290142, and 91630312), the research of
		L. Ji was supported by the NNSFC (NOs. 11601032, and 11471310).}

\maketitle	
	\begin{abstract}
%{In this paper we investigate the mean-square convergence of a semi-discrete numerical scheme for nonlinear stochastic Maxwell equations with multiplicative noise. The scheme is implicit in the drift term and explicit in the diffusion term of the equations, which is better suited to It\^{o} product. Allowing sufficient assumptions of the coefficients, nonlinear terms and the noise, we derive the regularities of the solution of stochastic Maxwell equations, including the uniform boundedness and H\"older continuity of the solution under several different norms. Furthermore, we prove that the numerical scheme has mean-square convergence order $\frac{1}{2}$.}\\
{In this paper, we propose a semi-implicit Euler scheme to discretize the stochastic nonlinear Maxwell equations with multiplicative It\^o noise, which is implicit in the drift term and explicit in the diffusion term of the equations, in order to suited to It\^{o} product. Uniform bounds with high regularities of solutions for both the continuous and the discrete problems are obtained, which are crucial properties to derive the mean-square convergence with certain order.  Allowing sufficient spatial regularity and utilizing the energy estimate technique, the convergence order $\frac12$ in mean-square sense is obtained.}\\
%{\sc AMS subject classification: }{\rm\small60H15, 35Q61}\\
	\end{abstract}

	\section{Introduction}
Stochastic Maxwell equations play an important role in stochastic electromagnetism and statistical radiophysics fields. Some articles (see, e.g., \cite{RKT1989,FM2008,Fann2007}) introduced  randomness into Maxwell equations in order to strengthen the correspondence between theoretical results and the real-life situations. In \cite{Mich2006}, problems about how to account, rigorously, for uncertainties in classical macroscopic electromagnetic interactions between fields and systems of linear material were discussed.
% On the other hand, finding the relationships (e.g. fluctuation-dissipation theorem) between the electromagnetic fields and the underlying randomness is also an important issue.
\cite{L2015} considered the problem about how to use the spectral representation to describe the random electromagnetic fields, which are coupled by Maxwell's equations with a random source term.
 \cite{RSY2012} dealt with the mathematical analysis of stochastic problems arising in the theory of electromagnetic in complex media, including well-posedness, controllability and homogenization.
Assuming the existence of magnetic charges or monopoles,
consider  the following generalized symmetrized stochastic nonlinear Maxwell equations driven by multiplicative It\^o noise,
\begin{equation}\label{sto_max}
\begin{cases}
\varepsilon \partial_{t}{\bf E}-\nabla\times {\bf H}=-{\bf J}_{e}(t,{\bf x},{\bf E},{\bf H})-{\bf J}_e^{r}(t,{\bf x},{\bf E},{\bf H})\cdot\dot{W},~ &(t,{\bf x})\in(0,~T]\times D,\\
\mu \partial_{t}{\bf H}+\nabla\times {\bf E}=-{\bf J}_{m}(t,{\bf x},{\bf E},{\bf H})-{\bf J}_m^{r}(t,{\bf x},{\bf E},{\bf H})\cdot\dot{W},~ &(t,{\bf x})\in(0,~T]\times D,\\
{\bf E}(0,{\bf x})={\bf E}_0({\bf x}),~{\bf H}(0,{\bf x})={\bf H}_0({\bf x}),~&{\bf x}\in D,\\
{\bf n}\times {\bf E}={\bf 0},~&(t,{\bf x})\in(0,~T]\times\partial D,
\end{cases}
\end{equation}
where $D\subset {\mathbb R}^{d}$ with $d=3$ is a bounded domain, $T\in(0,~\infty)$, and the function  ${\bf J}:[0,T]\times D\times {\mathbb R}^d\times{\mathbb R}^d\to{\mathbb R}^d$ is a smooth nonlinear function satisfying
\begin{align}\label{bound J}
|{\bf J}(t,{\bf x},u,v)|\leq L(1+|u|+|v|),
%&|{\bf J}(t,{\bf x},u_1,v_1)-{\bf J}(s,{\bf x},u_2,v_2)|\leq L(|t-s|+|u_1-u_2|+|v_1-v_2|),\label{bound partialJ}
\end{align}
and
\begin{equation}\label{bound partialJ}
|{\bf J}(t,{\bf x},u_1,v_1)-{\bf J}(s,{\bf x},u_2,v_2)|\leq L(|t-s|+|u_1-u_2|+|v_1-v_2|),
\end{equation}
for all ${\bf x}\in D$, $t,s\in[0,T]$, $u,v,u_1,v_1,u_2,v_2\in{\mathbb R}^d$ and some constant $L>0$. Here $|\cdot|$ denotes the Euclidean norm, and ${\bf J}$ could be ${\bf J}_e$, ${\bf J}_e^r$, ${\bf J}_m$ or ${\bf J}_m^r$.

Recently, more and more attention has been paid to the numerical analysis of stochastic Maxwell equations. In \cite{Zhang2008}, the author considered the stochastic Maxwell equations \eqref{sto_max} driven by a color noise and investigated the finite element method for these equations and furthermore obtained the $L^2$ error estimates. In \cite{BAZC2010}, the authors considered the two-dimensional Maxwell equations through a random source term and constructed a new numerical method based on Wiener chaos expansion. Due to the superiorities of multi-symplectic methods, many researchers have studied the stochastic multi-symplectic methods to stochastic Maxwell equations. \cite{HJZ2014} first proposed a stochastic multi-symplectic method for stochastic Maxwell equations with additive noise by using stochastic variational principle. The further analysis of preservation of physical properties of stochastic Maxwell equations with additive noise via stochastic multi-symplectic methods was investigated in \cite{CHZ2016}. More recently, \cite{HJZC2017} designed an innovative stochastic multi-symplectic method to three-dimensional stochastic Maxwell equations with multiplicative noise based on wavelet interpolation technique. This method has been applied successfully to solve
a three-dimensional stochastic electromagnetic fields problem with periodic boundary condition.

The main difficulty in dealing with stochastic partial differential equations is the presence of unbounded differential operators and stochastic integrals. Here we strongly use the fact that the equation is semilinear so that we can write it in the abstract form:
\begin{equation}\label{1.4}
\begin{cases}
{\rm d}u(t)=\left[Mu(t)+F(t,u(t))\right]{\rm d}t+B(t,u(t)){\rm d}W(t),~t\in(0,~T],\\
u(0)=u_0,
\end{cases}
\end{equation}
whose solution can be writen in an integral form  containing a bounded linear semigroup instead of the unbounded differential operator,
	\begin{equation}\label{1.5}
u(t)=S(t)u_0+\int_0^t S(t-s)F(s,u(s)){\rm d}s+\int_0^t S(t-s)B(s,u(s))dW(s)~ a.s.,
\end{equation}
with $M$ being the Maxwell operator and generating the unitary $C_0$-semigroup   $S(t)=e^{tM}$; see Section 2 for the procedure of rewriting stochastic Maxwell equations \eqref{sto_max} into the abstract form \eqref{1.4}.
 Most of the analysis is made on this mild solution form \eqref{1.5} of the equation. In this way we require the minimal regularity assumptions on the solutions.
We first establish the uniform boundedness of the solution in $L^{p}(\Omega;{\mathcal D}(M^k))$-norm for a given integer $k\in{\mathbb N}$ with ${\mathcal D}(M^k)$ being the $k$-th power of the operator $M$. Thanks to the mild solution \eqref{1.5} and the estimates for stochastic convolutions, we get
	\begin{eqnarray}\label{1.6}
\sup_{t\in[0,T]}\|u(t)\|_{L^{p}(\Omega;{\mathcal D}(M^k))}\leq C(1+\|u_0\|_{L^p(\Omega;{{\mathcal D}(M^k)})}),
\end{eqnarray}
where the positive constant $C$ may depend on $p$, $T$, and $\|Q^{\frac12}\|_{HS(U,H^{k+\gamma}(D))}$ with $\gamma >d/2$. After establishing the convergence order of the semigroup $S(t)$ and identity operator $Id$ with respect to time $t$ in Lemma \ref{est ope},
 the H\"older continuity of the solution in $L^{2}(\Omega;{\mathcal D}(M^{k-1}))$-norm is derived, i.e,
 	\begin{eqnarray}
 {\mathbb E}\|u(t)-u(s)\|_{{\mathcal D}(M^{k-1})}^2\leq C|t-s|,
 \end{eqnarray}
 where the positive constant $C$ may depend on $p$, $T$,  $\|Q^{\frac12}\|_{HS(U,H^{k+\gamma}(D))}$, and $\|u_0\|_{L^p(\Omega;{{\mathcal D}(M^k)})}$.

 The main goal of this work is to propose and study a semi-discretization in the temporal direction of \eqref{1.4} which inherits a uniform estimate in ${\mathcal D}(M)$-norm,
 \begin{equation}\label{1.8}
 u^{n+1}=u^{n}+\tau Mu^{n+1}+\tau F(t_{n+1},u^{n+1})+B(t_n,u^{n})\Delta W^{n+1}, ~n\geq 0.
 \end{equation}
 We show the existence of an $\{{\mathcal F}_{t_n};0\leq n\leq N\}$-adapted discrete solution  $\{u^n;~n\in{\mathbb N}\}$, and the iterates $\{u^n;~n\in{\mathbb N}\}$ satisfy
 	\begin{equation}\label{1.9}
 \max_{1\leq n\leq N}{\mathbb E}\|u^n\|^{p}_{{\mathcal D}(M)}\leq C(1+\|u_0\|^{p}_{L^{p}(\Omega;{\mathcal D}(M))}).
 \end{equation}
 In order to derive this result, we take ${\mathbb H}$-inner product of \eqref{1.8} with $u^{n+1}$ and $Mu^{n+1}-Mu^{n}$, respectively. It is
 important to note that the appearance of terms $\|u^{n+1}-u^{n}\|^2_{\mathbb H}$ in the right-hand side of \eqref{4.5} and $\|Mu^{n+1}-Mu^{n}\|^2_{\mathbb H}$ in the right-hand side of \eqref{4.7} is crucial to obtain the boundedness result \eqref{1.9}, which could absorb the difficulty cased by the stochastic term.
It shows  the stability of the iterates $\{u^n;~n\in{\mathbb N}\}$ for the scheme \eqref{1.8}.

 It is  important to understand how the numerical methods
approximate the solutions of \eqref{sto_max} and the first step is to analyze the error.
 In the second part of this work, we are interested in the mean-square convergence for iterates $\{u^n;~n\in{\mathbb N}\}$ of \eqref{1.8}.
 To the best of our knowledge, however, there has been no work in the literature which analyzes the convergence of numerical method for stochastic
 Maxwell equations.
 A relevant prerequisite for this purpose is to provide strong stability results \eqref{1.6} for the original problem \eqref{1.4}, and \eqref{1.9} for  the discretization \eqref{1.8}.
% In this paper, we wish to make precise the convergence and estimate the order of the error in mean-square sense for a semi-implicit Euler discretization to stochastic nonlinear Maxwell equations \eqref{sto_max}. It is easy to see that this requires more regularity on the solutions, see \eqref{1.6}.
%  For the discrete problem,  the existence of a ${\mathcal D}(M)$-valued $\{{\mathcal F}_{t_n}\}_{0\leq n\leq N}$-adapted solution is another crucial property to prove the convergence order.
Define the local truncation error of scheme \eqref{1.8} by
\begin{equation}\label{1.10}
\delta^{n+1}:=u(t_{n+1})-u(t_n)-\tau Mu(t_{n+1})-\tau F(t_{n+1},u(t_{n+1}))-B(t_n,u(t_n))\Delta W^{n+1},
\end{equation}	
where $u(t)$ means the solution of stochastic nonlinear Maxwell equations \eqref{1.4}.
Using the classical energy technique, we find the relationship between the global error in mean-square sense and the local truncation error in mean and mean-square senses, i.e.,
\begin{equation}
{\mathbb E}\|e^{n+1}\|^2_{\mathbb H}\leq {\mathbb E}\|e^{n}\|^2_{\mathbb H}+C\tau\big({\mathbb E}\|e^{n}\|^2_{\mathbb H}+{\mathbb E}\|e^{n+1}\|^2_{\mathbb H}\big)+C{\mathbb E}\|\delta^{n+1}\|^2_{\mathbb H}
+\frac{C}{\tau}{\mathbb E}\|{\mathbb E}(\delta^{n+1}|{\mathcal F}_{t_n})\|^2_{\mathbb H},
\end{equation}
where $e^n=u(t_n)-u^n$.
It states that the global mean-square convergence order depends only on the local truncation error in mean and mean-square senses for sufficiently small time step size.
Via replacing the expression of $u(t_{n+1})-u(t_n)$ in \eqref{1.10} by the strong solution of \eqref{1.4}, we get the estimates of local truncation error $\delta^{n+1}$,
	\begin{equation}
\mathbb{E}\|\delta^{n+1}\|_{\mathbb H}^2\leq C\tau^2,\quad{\mathbb E}\|\mathbb{E}\big(\delta^{n+1}|{\mathcal F}_{t_{n}}\big)\|^2_{\mathbb H}\leq C\tau^{3},
\end{equation}
which leads to
	\begin{equation}\label{1.13}
\max_{0\leq n\leq N} \left(\mathbb{E}\|e^n\|_{\mathbb H}^2\right)^{1/2}\leq C\tau^{\frac12},
\end{equation}
where the positive constant $C$ may depend on the Lipschitz coefficients of $F$ and $B$, $T$, $\|u_0\|_{L^2(\Omega;{\mathcal D}(M^2))}$ and $\|Q^{\frac12}\|_{HS(U,H^{2+\gamma}(D))}$, but independent of $\tau$ and $n$.
The estimate \eqref{1.13} means that   the mean-square convergence order of \eqref{1.8} is of $1/2$.

%In this paper, we consider a semi-implicit discretization in temporal direction for stochastic nonlinear Maxwell equations. The mean-square convergence order $1/2$ is derived allowing sufficient spatial regularity assumption on the noise.
%

%The full discretization of the stochastic Maxwell equations will be dealt with in a future work. Again, the weak or strong order of the error depends on the smoothness of the solutions. It also depends on the spatial discretization. Provided enough smoothness is required on the noise covariance and on the initial data, it is expected that the strong order in space is the same as in the deterministic case. If a semidescritization in space is considered, it is probable that the weak and strong  order are the same.

The outline of this paper is as follows. In Section 2, some preliminaries are collected and an abstract formulation of \eqref{sto_max} is set forth. In Section 3, we analyze the regularities of the solution of stochastic nonlinear Maxwell equations \eqref{sto_max}, including the uniform boundedness and H\"older continuity. In Section 4, a semi-implicit Euler scheme in temporal direction is proposed and the mean-square convergence order is derived.

\section{Preliminaries and framework}
For the coefficients of equations \eqref{sto_max} we suppose that
\begin{equation}
\mu, \varepsilon\in L^{\infty}(D),~\mu,\varepsilon\geq \delta>0,
\end{equation}
for a constant $\delta>0$. The basic Hilbert space we work with is ${\mathbb H}=L^2(D)^3\times L^2(D)^3$ with the inner product is defined by
\[
\left\langle \begin{pmatrix}
{\bf E}_1\\{\bf H}_1
\end{pmatrix},~ \begin{pmatrix}
{\bf E}_2\\{\bf H}_2
\end{pmatrix}\right\rangle_{\mathbb H}:=\int_{D}(\varepsilon {\bf E}_1\cdot {\bf E}_2
+\mu{\bf H}_1\cdot{\bf H}_2){\rm d}{\bf x}.
\]
By our assumption on $\mu$ and $\varepsilon$, this inner product is obviously equivalent to the standard inner product on $L^2(D)^6$. The norm induced by this inner product corresponds to the electromagnetic energy of the physical system
\[
\left\|\begin{pmatrix}
{\bf E}\\{\bf H}
\end{pmatrix}\right\|_{\mathbb H}=\left(\int_{D}(\varepsilon|{\bf E}|^2+\mu|{\bf H}|^2){\rm d}{\bf x}\right)^{\frac{1}{2}}.
\]
If there's no external source, the electromagnetic energy of \eqref{1.4} is a conserved quantity, i.e., $\|u(t)\|_{\mathbb H}=\|u_0\|_{\mathbb H}$.

The Maxwell operator is defined by
\begin{equation}\label{M_operator}
M=\left(
\begin{array}{cc}
0& \varepsilon^{-1}\nabla\times \\
-\mu^{-1}\nabla\times &0 \\
\end{array}
\right)
\end{equation}
with domain
\begin{equation}\label{domain}
\begin{split}
{\mathcal D}(M)&=\left\{ \left(\begin{array}{c}
{\bf E} \\
{\bf H}
\end{array}\right)\in {\mathbb H}:~M\left(\begin{array}{c}
{\bf E} \\
{\bf H}
\end{array}\right)=\left(\begin{array}{c}
\varepsilon^{-1} \nabla\times{\bf H}\\
-\mu^{-1}\nabla\times{\bf E}
\end{array} \right)\in{\mathbb H},~ {\bf n}\times{\bf E}\Big|_{\partial D}={\bf 0} \right\}\\
&=H_0({\rm curl},D)\times H({\rm curl},D),
\end{split}
\end{equation}
where the curl-spaces are defined by
\[
H({\rm curl},D):=\{ v\in L^2(D)^3:~\nabla\times v\in L^2(D)^3 \},
\]
and
\[
H_0({\rm curl},D):=\{ v\in H({\rm curl},D):~{\bf n}\times v|_{\partial D}={\bf 0} \}.
\]
The corresponding graph norm is $\|v\|_{{\mathcal D}(M)}:=\big(\|v\|_{\mathbb H}^2+\|Mv\|_{\mathbb H}^2\big)^{1/2}$.
The Maxwell operator $M$ defined in \eqref{M_operator} with domain \eqref{domain} is colsed,  skew-adjoint on $\mathbb{H}$, and thus  generates a unitary $C_0$-group $S(t)=e^{tM}$ on $\mathbb{H}$ in the view of Stone's theorem (see for instance \cite[Theorem II.3.24]{EN2000}). A frequently used property for Maxwell operator $M$ is:
$
\langle Mu,~u\rangle_{\mathbb H}=0, ~\forall~u\in{\mathcal D}(M).
$
We refer interested readers to \cite[Chapter 3]{Pzur2013} and references therein for more introduction about Maxwell operator.
Recursively, we could define the domain ${\mathcal D}(M^k)=\{u\in{\mathcal D}(M^{k-1}):~M^{k-1}u\in{\mathcal D}(M)\}$ for the $k$-th power of the operator $M$, $k\in{\mathbb N}$, with ${\mathcal D}(M^0)={\mathbb H}$, given the norm
\[
\|v\|_{{\mathcal D}(M^k)}:=\Big(\|v\|_{\mathbb H}^2+\|M^k v\|_{\mathbb H}^2\Big)^{1/2},~\forall~v\in {\mathcal D}(M^k),
\]
which is a Hilbert space. In fact,
the norm $\|\cdot\|_{{\mathcal D}(M^k)}$ corresponds to the scalar product
\[
\langle u,~v\rangle_{{\mathcal D}(M^k)}=\langle u,~v\rangle_{\mathbb H}
+\langle M^k u,~M^k v\rangle_{\mathbb H}, ~\forall~u,v\in {\mathcal D}(M^k),
\]
and thus ${\mathcal D}(M^k)$ is a pre-Hilbert space.
If $\{v_{\ell}\}_{\ell\in{\mathbb N}}$  is a Cauchy sequence for $\|\cdot\|_{{\mathcal D}(M^k)}$, then $\{v_{\ell}\}_{\ell\in{\mathbb N}}$ and $\{M^k v_{\ell}\}_{\ell\in{\mathbb N}}$ are Cauchy sequences in ${\mathbb H}$.
Since ${\mathbb H}$ is complete, $v_{\ell}\to v$ and $M^k v_{\ell}\to v^{k}$ in ${\mathbb H}$. The closedness of operator $M$ leads to $v^{k}=M^k v$, i.e., $v_{\ell}\to v$ in ${\mathcal D}(M^k)$ which is thus a Hilbert space.
Moreover, it can be shown that $\|u\|_{{\mathcal D}(M^{k_1})}\leq C\|u\|_{{\mathcal D}(M^{k_2})}$, $\forall~u\in{\mathcal D}(M^{k_2})$, $k_1\leq k_2$.

Let $F:[0,~T]\times {\mathbb H}\to{\mathbb H}$ be a Nemytskij operator associated to ${\bf J}_{e}$, ${\bf J}_m$, defined by
\begin{equation}
F(t,u)({\bf x})=\left( \begin{array}{c}
-\varepsilon^{-1}{\bf J}_{e}(t,{\bf x},{\bf E}(t,{\bf x}),{\bf H}(t,{\bf x}))\\
-\mu^{-1}{\bf J}_{m}(t,{\bf x},{\bf E}(t,{\bf x}),{\bf H}(t,{\bf x}))
\end{array} \right), ~{\bf x}\in D,~u=\left(\begin{array}{c}
{\bf E} \\
{\bf H}
\end{array}\right)\in{\mathbb H}.
\end{equation}
\iffalse
The derivative operators of $F$ is given by
\begin{equation}
F^{\prime}(t,u)(v)({\bf x})=\left( \begin{array}{c}
-\varepsilon^{-1}{\partial_{1} {\bf J}_{e}(t,{\bf x},{\bf E}_1(t,{\bf x}),{\bf H}_1(t,{\bf x}))}{\bf E}_2(t,{\bf x})
-\varepsilon^{-1}{\partial_{2} {\bf J}_{e}(t,{\bf x},{\bf E}_1(t,{\bf x}),{\bf H}_1(t,{\bf x}))}{\bf H}_2(t,{\bf x})\\
-\mu^{-1}\partial_1{\bf J}_{m}(t,{\bf x},{\bf E}_1(t,{\bf x}),{\bf H}_1(t,{\bf x})){\bf E}_2(t,{\bf x})
-\mu^{-1}\partial_1{\bf J}_{m}(t,{\bf x},{\bf E}_1(t,{\bf x}),{\bf H}_1(t,{\bf x})){\bf H}_2(t,{\bf x})
\end{array} \right),
\end{equation}
where
${\bf x}\in D,~u=\left(\begin{array}{c}
{\bf E}_1 \\
{\bf H}_1
\end{array}\right)\in{\mathbb H},~\text{and }
v=\left(\begin{array}{c}
{\bf E}_2 \\
{\bf H}_2
\end{array}\right)\in{\mathbb H}$.
\fi
Thanks to \eqref{bound J} and \eqref{bound partialJ}, the operator $F$ satisfies
\begin{eqnarray}
&&\|F(t,u)\|_{\mathbb H} \leq C(1+\|u\|_{\mathbb H}),\label{bound F}\\
&&\|F(t,u)-F(s,v)\|_{\mathbb H}\leq C (|t-s|+\|u-v\|_{\mathbb H}),\label{Lip F}
\end{eqnarray}
for all $t,s\in[0,T]$, and $u,v\in{\mathbb H}$. Here the positive constant $C$
may depend on  $\delta$, the volume $|D|$ of domain $D$, and the constant $L$ in \eqref{bound J} and \eqref{bound partialJ}. In fact,
\begin{equation*}
\begin{split}
\|F(t,u)\|_{\mathbb H}&=\Big(\int_{D}\varepsilon|\varepsilon^{-1}{\bf J}_e|^2
+\mu|\mu^{-1}{\bf J}_m|^2 {\rm d}{\bf x}\Big)^{\frac12}\\
&\leq \delta^{-\frac12}\Big(\int_{D}2L^2(1+|{\bf E}|+|{\bf H}|)^2{\rm d}{\bf x}\Big)^{\frac12}\\
&\leq \delta^{\frac12}\Big[(6L^2|D|)^{\frac12}+\Big(6L^2\delta^{-1}\int_{D}(\varepsilon|{\bf E}|^2+\mu|{\bf H}|^2){\rm d}{\bf x}\Big)^{\frac12} \Big]\\
&\leq C(1+\|u\|_{\mathbb H}),
\end{split}
\end{equation*}
and the proof of \eqref{Lip F} is similar as above.

Let $Q$ be a symmetric, positive definite operator with finite trace. The driven stochastic process $W(t)$ is a standard $Q$-Wiener process with respect to the filtered probability space $(\Omega,{\mathcal F},\{{\mathcal F}_{t}\}_{0\leq t\leq T},{\mathbb P})$, which can be represented as
\begin{equation}
W(t)=\sum_{i=1}^{\infty}Q^{\frac12}e_i\beta_i(t),~t\in[0,~T],
\end{equation}
where $\{\beta_i(t)\}_{i\in{\mathbb N}}$ is a family of  independent real-valued Brownian motions and $\{e_i\}_{i\in{\mathbb N}}$ is an orthonormal basis of the space $U=L^2(D)$.

For diffusion term, we introduce the Nemytskij operator $B:[0,~T]\times{\mathbb H}\to HS(U_0,{\mathbb H})$ by
\begin{equation}
(B(t,u)v)({\bf x})=\left( \begin{array}{c}
-\varepsilon^{-1}{\bf J}_{e}^{r}(t,{\bf x},{\bf E}(t,{\bf x}),{\bf H}(t,{\bf x}))v({\bf x})\\
-\mu^{-1}{\bf J}_{m}^r(t,{\bf x},{\bf E}(t,{\bf x}),{\bf H}(t,{\bf x}))v({\bf x})
\end{array} \right),
\end{equation}
where ${\bf x}\in D,~u=\left(\begin{array}{c}
{\bf E} \\
{\bf H}
\end{array}\right)\in{\mathbb H}, \text{ and } v\in U_0:=Q^{\frac12}U$. Here $HS(U,H)$ denotes the separable Hilbert space of all Hilbert-Schmidt operators from one separable Hilbert space $U$ to another separable Hilbert space $H$, equipped with the scalar product
\[
\langle \Gamma_1,~\Gamma_2 \rangle_{HS(U,H)}=\sum_{j=1}^{\infty}\langle \Gamma_1 \eta_j,~\Gamma_2 \eta_j \rangle_{H},
\]
and the corresponding norm
\[
\|\Gamma\|_{HS(U,H)}=\left(\sum_{j=1}^{\infty}\|\Gamma\eta_j\|_{H}^2\right)^{\frac12},
\]
where $\{\eta_j\}_{j\in{\mathbb N}}$ is an orthonormal basis of $U$.

Thanks to \eqref{bound J} and \eqref{bound partialJ}, we have
\begin{eqnarray}
&&\|B(t,u)\|_{HS(U_0,{\mathbb H})}\leq C \|Q^{\frac12}\|_{HS(U,H^{\gamma}(D))}(1+\|u\|_{\mathbb H}^2)^{\frac12},\label{bound B}\\
&&\|B(t,u)-B(s,v)\|_{HS(U_0,{\mathbb H})}\leq C\|Q^{\frac12}\|_{HS(U,H^{\gamma}(D))}(|t-s|+\|u-v\|_{\mathbb H}),\label{Lip B}
\end{eqnarray}
for all $t,s\in[0,~T]$, $u,v\in {\mathbb H}$. Here the positive constant $C$
may depend on  $\delta$, the volume $|D|$ of domain $D$, and the constant $L$ in \eqref{bound J} and \eqref{bound partialJ}. In fact,
\begin{equation*}
\begin{split}
&\|B(t,u)\|_{HS(U_0,{\mathbb H})}^2=\|B(t,u)Q^{\frac12}\|^2_{HS(U,{\mathbb H})}
=\sum_{j=1}^{\infty}\|B(t,u)Q^{\frac12}e_j\|_{\mathbb H}^2\\
&=\sum_{j=1}^{\infty}\int_{D} \varepsilon^{-1}|{\bf J}_e^rQ^{\frac12}e_j({\bf x})|^2+\mu^{-1}|{\bf J}_m^rQ^{\frac12}e_j({\bf x})|^2 {\rm d}{\bf x}\\
&\leq 6L^2\delta^{-1} \sum_{j=1}^{\infty}\|Q^{\frac12}e_j\|_{L^{\infty}(D)}^2
\int_{D}(1+|{\bf E}|^2+|{\bf H}|^2){\rm d}{\bf x}\\
&\leq 6L^2\delta^{-1} \|Q^{\frac12}\|^2_{HS(U,H^{\gamma}(D))}(|D|+\delta^{-1}\int_{D}\varepsilon|{\bf E}|^2+\mu|{\bf H}|^2{\rm d}{\bf x})\\
&\leq C \|Q^{\frac12}\|^2_{HS(U,H^{\gamma}(D))}(1+\|u\|_{\mathbb H}^2),
\end{split}
\end{equation*}
where we have used the Sobolev embedding $H^{\gamma}(D)\subset L^{\infty}(D)$ for any $\gamma>d/2$, $d=3$ in this paper.
and the proof of \eqref{Lip B} is similar as above.

At this point, we introduce the abstract form of stochastic nonlinear Maxwell equations in infinite dimensional space ${\mathbb H}$:
\begin{equation}\label{sM_equations}
\begin{cases}
{\rm d}u(t)=\left[Mu(t)+F(t,u(t))\right]{\rm d}t+B(t,u(t)){\rm d}W(t),~t\in(0,~T],\\
u(0)=u_0,
\end{cases}
\end{equation}
where $M$, $F$, $B$ and $W$ are defined as above, and
\[
u=\begin{pmatrix}{\bf E}(t,{\bf x})\\ {\bf H}(t,{\bf x})\end{pmatrix},\quad
u_0=\begin{pmatrix}{\bf E}_0({\bf x})\\ {\bf H}_0({\bf x})\end{pmatrix}.
\]
Now we look at the existence and uniqueness of the mild solution of stochastic nonlinear Maxwell equations \eqref{sto_max} under certain conditions on the original functions $J_{e}, J_{m},J_{e}^r,J_m^r$, operator $Q$ and initial data;
see \cite{RSY2012} for the well-posedness of stochastic Maxwell equations  in complex media given conditions directly on $F$ and $B$. Moreover, using the Burkholder-Davis-Gundy-type inequality we present an priori estimation on $\sup_{t\in[0,T]}\|u(t)\|_{L^{p}(\Omega;{\mathbb H})}$ in Theorem \ref{well-posedness} and on ${\mathbb E}\Big(\sup_{t\in[0,T]}\|u(t)\|^{p}_{\mathbb H}\Big)$ in Corollary \ref{well-posedness2}; see \cite{PZ2014} for the estimations about stochastic integrals and stochastic convolutions.
\begin{theorem}\label{well-posedness}
	Suppose conditions \eqref{bound J} and \eqref{bound partialJ} are fulfilled, let $W(t)$, $t\in[0,~T]$ be a $Q$-Wiener process with $Q^{\frac12}\in HS(U,H^{\gamma}(D))$ for $\gamma>d/2$, and let $u_0$ be an ${\mathcal F}_0$-measurable ${\mathbb H}$-valued random variable satisfying $\|u_0\|_{L^{p}(\Omega;{\mathbb H})}<\infty$ for some $p\geq 2$. Then stochastic Maxwell equations \eqref{sM_equations} have a unique mild solution given by
	\begin{equation}\label{mild sol}
	u(t)=S(t)u_0+\int_0^t S(t-s)F(s,u(s)){\rm d}s+\int_0^t S(t-s)B(s,u(s))dW(s),~ {\mathbb P}\text{-}a.s.,
	\end{equation}
	for each $t\in[0,~T]$.
	Moreover, there exists a constant $C:=C(p,T,\|Q^{\frac12}\|_{HS(U,H^{\gamma}(D))})\in(0,~\infty)$ such that
	\begin{align}
	\sup_{t\in[0,T]}\|u(t)\|_{L^{p}(\Omega;{\mathbb H})}\leq C(1+\|u_0\|_{L^{p}(\Omega;{\mathbb H})}).
%		{\mathbb E}\Big(\sup_{t\in[0,T]}\|u(t)\|^{p}_{\mathbb H}\Big)\leq C(1+\|u_0\|_{L^{p}(\Omega;{\mathbb H})}).
	\end{align}
\end{theorem}
\begin{proof}
	Under conditions  \eqref{bound J} and \eqref{bound partialJ}, we see that from \eqref{Lip F} and \eqref{Lip B}
 $F$ and $B$ are both globally Lipschitz functions, the existence and uniqueness of the mild solution  \eqref{mild sol} follows from \cite[Theorem 12.4.7]{RSY2012}, or \cite[Theorem 7.2]{PZ2014} for general stochastic evolution equation. Using the Burkholder-Davis-Gundy-type inequality \cite[Theorem 4.36]{PZ2014}, linear growth \eqref{bound F} and \eqref{bound B} of $F$ and $B$, we have
	\begin{equation*}
	\begin{split}
	{\mathbb E}\|u(t)\|^{p}_{{\mathbb H}}	\lesssim& {\mathbb E}\|S(t)u_0\|^{p}_{{\mathbb H}}
	+{\mathbb E}\int_0^t\|S(t-s)F(s,u(s))\|^{p}_{{\mathbb H}}{\rm d}s\\
	&	+{\mathbb E}\Big\|\int_0^t S(t-s)B(s,u(s)){\rm d}W(s)\Big\|^{p}_{{\mathbb H}}\\
	\lesssim& {\mathbb E}\|u_0\|^{p}_{{\mathbb H}}+\int_0^t (1+{\mathbb E}\|u(s)\|^{p}_{{\mathbb H}}){\rm d}s
	+\Big[{\mathbb E}\int_0^t\|B(s,u(s))\|_{HS(U_0,{\mathbb H})}^{2}{\rm d}s\Big]^{\frac{p}{2}}\\
	\lesssim & \|u_0\|^p_{L^p(\Omega;{\mathbb H})}+\int_0^t (1+{\mathbb E}\|u(s)\|^{p}_{{\mathbb H}}){\rm d}s+\|Q^{\frac12}\|^{p}_{HS(U,H^{\gamma}(D))}\int_0^t (1+{\mathbb E}\|u(s)\|^{p}_{{\mathbb H}}){\rm d}s,
	\end{split}
	\end{equation*}
	where notation $A\lesssim B$ means that there exists a positive constant such that $A\leq CB$.
	
	By Gronwall's inequality, there exists a positive constant $C:=C(p,T,\|Q^{\frac12}\|_{HS(U,H^{\gamma}(D))})$ such that
	\[
	{\mathbb E}\|u(t)\|_{\mathbb H}^p\leq C(1+\|u_0\|^p_{L^p(\Omega;{\mathbb H})}),~\forall ~t\in[0,T].
	\]
	Therefore, we complete the proof.
\end{proof}

\begin{corollary}\label{well-posedness2}
	Under the same conditions of Theorem \ref{well-posedness}, there exists a constant $C:=C(p,T,u_0,Q)$ such that
		\begin{align}\label{est sup}
	{\mathbb E}\Big(\sup_{t\in[0,T]}\|u(t)\|^{p}_{\mathbb H}\Big)\leq C.
	\end{align}
\end{corollary}
\begin{proof}
	The main step to derive \eqref{est sup} from the mild solution \eqref{mild sol} is that we need to deal with the stochastic integral
	\[
	{\mathbb E}\left[\sup_{t\in[0,T]}\left\| \int_0^t S(t-s)B(s,u(s)){\rm d}W(s) \right\|^{p}_{{\mathbb H}}\right].
	\]
	By using the Burkholder-Davis-Gundy-type inequality for stochastic convolution \cite[Proposition 7.3]{PZ2014}, we have
	\begin{equation*}
	\begin{split}
	{\mathbb E}\left[\sup_{t\in[0,T]}\left\| \int_0^t S(t-s)B(s,u(s)){\rm d}W(s) \right\|^{p}_{{\mathbb H}}\right]
	&\lesssim {\mathbb E}\int_{0}^{T}\|S(t-s)B(s,u(s))\|^{p}_{HS(U_0,{\mathbb H})}{\rm d}s\\
&	\lesssim \|Q^{\frac12}\|^{p}_{HS(U,H^{\gamma}(D))}\int_0^t (1+{\mathbb E}\|u(s)\|^{p}_{{\mathbb H}}){\rm d}s\leq C(p,T,u_0,Q),
	\end{split}
	\end{equation*}
	where we use the result of Theorem \ref{well-posedness} in the last step.
\end{proof}

\begin{remark}
	If we apply It\^o formula to the functional  ${\mathcal H}(u)=\|u\|^2_{\mathbb H}$, we may get the evolution of the electromagnetic energy of system \eqref{sM_equations}. In fact, the first and the second order derivatives of ${\mathcal H}(u)$ are
	\[D{\mathcal H}(u)(\psi)=2\langle u,\psi\rangle_{{\mathbb H}},
	\quad
	D^2{\mathcal H}(u)(\psi,\phi)=2\langle \psi,\phi\rangle_{{\mathbb H}}.
	\]
	It\^o formula (see for instance \cite[Theorem 4.32]{PZ2014}) gives us
	\begin{equation}
	\begin{split}
	{\mathcal H}(u(t))=&{\mathcal H}(u_0)+\int_{0}^{t}2\langle u(s),F(s,u(s))\rangle_{{\mathbb H}}+tr\big[\langle B(s,u(s))Q^{\frac12},(B(s,u(s))Q^{\frac12})^{*}\rangle_{{\mathbb H}}\big]{\rm d}s\\
	&+2\int_0\langle u(s),B(s,u(s)){\rm d}W(s)\rangle_{{\mathbb H}}.
	\end{split}
	\end{equation}
	We observe that if $F=0$ and $B$ is a constant operator, then the average energy ${\mathbb E}{\mathcal H}(u(t))$ grows linearly with respect to time $t$, see \cite[Theorem 2.1]{CHZ2016} for the analysis of stochastic Maxwell equations with additive noise.
\end{remark}

\section{Regularities of the solution of stochastic Maxwell equations}
This section is devoted to the regularity analysis for the solution of stochastic Maxwell equations \eqref{sto_max} or \eqref{sM_equations}, including the uniform boundedness of the solution in $L^{p}(\Omega;{\mathcal D}(M^k))$-norm and H\"older continuity of the solution in $L^{2}(\Omega;{\mathcal D}(M^{k-1}))$-norm for a given  integer $k\in{\mathbb N}$.

First, we present the assumptions on coefficients of equations \eqref{sto_max} in order to get enough space regularity of the solution.

\begin{assumption}\label{assump0}
	Assume the coefficients $\mu,\varepsilon\in C_{b}^{k}(D)$, and $\mu,\varepsilon\geq \delta>0$.
\end{assumption}
By this assumption, we know that  for any integer $\ell\leq k$,
\begin{eqnarray}
&&\|\partial^{\ell}_{\bf x}\varepsilon\|_{L^{\infty}(D)}+\|\partial^{\ell}_{\bf x}\mu\|_{L^{\infty}(D)}\leq K_{1},\label{bound coeffi}\\
&&\|\partial^{\ell}_{\bf x}\varepsilon^{-1}\|_{L^{\infty}(D)}+
\|\partial^{\ell}_{\bf x}\mu^{-1}\|_{L^{\infty}(D)}\leq K_2,\label{bound deri coeff}
\end{eqnarray}
where $K_2$ depends on $\delta$ and $K_1$.

\begin{assumption}\label{assump}
	Assume function ${\bf J}:[0,T]\times D\times {\mathbb R}^d\times {\mathbb R}^d\to {\mathbb R}^d$ is a smooth enough nonlinear function with bounded derivatives, i.e.,
	${\bf J}\in C_b^{1,k,k+1,k+1}([0,T]\times D\times {\mathbb R}^d\times {\mathbb R}^d;{\mathbb R}^d)$.
	%	\begin{eqnarray}
	%	&&|(\nabla\times)^\ell{\bf J}(t,{\bf x},u({\bf x}),v(\bf{x})|\leq L(1+|u|+|v|),\label{bound J}\\
	%	%\Big|\frac{\partial {\bf J}}{\partial u}(t,{\bf x},u,v)\Big|+\Big|\frac{\partial {\bf J}}{\partial v}(t,{\bf x},u,v)\Big|\leq L,\label{bound partialJ}\\
	%	&&|{\bf J}(t,{\bf x},u_1,v_1)-{\bf J}(t,{\bf x},u_2,v_2)|\leq L(|u_1-u_2|+|v_1-v_2|),\label{bound partialJ}
	%	\end{eqnarray}
	%	for all ${\bf x}\in D$, $u,v,u_1,v_1,u_2,v_2\in{\mathbb R}^3$ and some constant $L>0$. Here $|\cdot|$ denotes the Euclidean norm, and
	Here ${\bf J}$ could be ${\bf J}_e$, ${\bf J}_e^r$, ${\bf J}_m$ or ${\bf J}_m^r$.
\end{assumption}

Throughout this paper $C_b^{\ell,m,n,n}$ denotes the set of vector-valued continuously differential functions $\Phi:(t,{\bf x},u,v)\in[0,T]\times D\times{\mathbb R}^d\times{\mathbb R}^d\to{\mathbb R}^d$ with uniformly bounded partial derivatives $\partial_{t}^{\ell_1}\Phi$, $\partial_{\bf x}^{m_2}\Phi$ and $\partial_{u}^{n_1}\partial_{v}^{n_2}\Phi$ for $\ell_1\leq\ell$, $m_1\leq m$ and $n_1+n_2\leq n$.

\begin{assumption}\label{assump 2}
	Assume that the operator $Q^{\frac12}\in HS(U,H^{k+\gamma}(D))$.
\end{assumption}

It follows from Assumptions \ref{assump0} and \ref{assump} that the drift term
$F$ satisfies
\begin{eqnarray}
&&\|F(t,u)\|_{{\mathcal D}(M^{\ell})}\leq C(1+\|u\|_{{\mathcal D}(M^{\ell})}),\label{linear growth of F Ml}\\
&&\|F(t,u)-F(s,v)\|_{{\mathcal D}(M^{\ell})}\leq C(|t-s|+\|u-v\|_{{\mathcal D}(M^{\ell})}),\label{Lip of F Ml}
\end{eqnarray}
where  $0\leq \ell\leq k$,  and $u,v\in {\mathcal D}(M^{\ell})$.
Here the constant $C$ may depend on $\delta$, $K_1$, $K_2$ in \eqref{bound coeffi} and \eqref{bound deri coeff},
the volume $|D|$ of the domain $D$, and the derivative bounds $L$ of functions ${\bf J}_e$ and ${\bf J}_m$. We only present the proof of \eqref{linear growth of F Ml} in the case $\ell=1$, for other cases and inequality \eqref{Lip of F Ml} could be proved by the same approach,
\begin{equation*}
\begin{split}
\|MF(t,u)\|_{{\mathbb H}}=&\Big(\int_{D}\varepsilon^{-1}|\nabla\times(\mu^{-1}{\bf J}_{m})|^2{\rm d}{\bf x}+\mu^{-1}|\nabla\times(\varepsilon^{-1}{\bf J}_{e})|^2{\rm d}{\bf x}\Big)^{\frac12}\\
\leq& \delta^{-\frac12}\Big[\int_{D}\delta^{-2}\big(|\nabla\times{\bf J}_m|^2+|\nabla\times{\bf J}_{e}|^2\big)+K_2^2\big(|{\bf J}_m|^2+|{\bf J}_e|^2\big){\rm d}{\bf x}\Big]^{\frac12}\\
\leq& \delta^{-\frac32}\Big[(6L^2|D|)^{\frac12}+\Big(6L^2K_1\int_{D}\mu^{-1}|\nabla\times{\bf E}|^2+\varepsilon^{-1}|\nabla\times{\bf H}|^2{\rm d}{\bf x}\Big)^{\frac12}\Big]\\
&+\delta^{-\frac12}K_2\Big[(6L^2|D|)^{\frac12}+\Big(6L^2\delta^{-1}\int_{D}\varepsilon|{\bf E}|^2+\mu|{\bf H}|^2{\rm d}{\bf x}\Big)^{\frac12}\Big]\\
\leq & C(1+\|u\|_{{\mathcal D}(M)}).
\end{split}
\end{equation*}

Under Assumptions \ref{assump0}, \ref{assump} and \ref{assump 2}, we get that for diffusion term $B$ and $0\leq \ell\leq k$, and for any $\gamma> d/2$,
\begin{eqnarray}
&&\|B(t,u)\|_{HS(U_0,{\mathcal D}(M^{\ell})}\leq C\|Q^{\frac12}\|_{HS(U,H^{\ell+\gamma}(D))}(1+\|u\|^2_{{\mathcal D}(M^{\ell})})^{\frac12},\label{linear growth of B Ml}\\
&&\|B(t,u)-B(s,v)\|_{HS(U_0,{\mathcal D}(M^{\ell})}\leq C\|Q^{\frac12}\|_{HS(U,H^{\ell+\gamma}(D))}(|t-s|+\|u-v\|_{{\mathcal D}(M^{\ell})}),\label{Lip of B Ml}
\end{eqnarray}
where $t,s\in[0,~T]$, and $u,v\in {\mathcal D}(M^{\ell})$.
Here the constant $C$ may depend on $\delta$, $K_1$, $K_2$ in \eqref{bound coeffi} and \eqref{bound deri coeff},
the volume $|D|$ of the domain $D$, and the derivative bounds $L$ of functions ${\bf J}^r_e$ and ${\bf J}^r_m$. We just present the proof of \eqref{linear growth of B Ml} in the case $\ell=1$, for other cases and inequality \eqref{Lip of B Ml} could be proved by the same approach,
\begin{align*}
%\begin{split}
\sum_{j=1}^{\infty}&\|M(BQ^{\frac12}e_j)\|^2_{\mathbb H}\\
=&\sum_{j=1}^{\infty}\int_{D}\varepsilon^{-1}|\nabla\times(\mu^{-1}{\bf J}_m^rQ^{\frac12}e_j)|^2+\mu^{-1}|\nabla\times(\varepsilon^{-1}{\bf J}_e^rQ^{\frac12}e_j)|^2{\rm d}{\bf x}\\
\leq & \delta^{-1}\sum_{j=1}^{\infty}\int_{D}\delta^{-2}\|Q^{\frac12}e_j\|^2_{L^{\infty}(D)}\big(|\nabla\times{\bf J}_{m}^{r}|^2+|\nabla\times{\bf J}_{e}^{r}|^2\big){\rm d}{\bf x}\\
&+\delta^{-1}\sum_{j=1}^{\infty}\int_{D}(K_2^2\|Q^{\frac12}e_j\|^2_{L^{\infty}(D)}+\delta^{-2}\|\nabla Q^{\frac12}e_j\|^2_{L^{\infty}(D)})\big(|{\bf J}_{m}^{r}|^2+|{\bf J}_{e}^{r}|^2\big){\rm d}{\bf x}\\
\leq & 6L^2\delta^{-3}\|Q^{\frac12}\|^2_{HS(U,H^{\gamma}(D))}\int_{D}(1+|\nabla\times{\bf E}|^2+|\nabla\times{\bf H}|^2){\rm d}{\bf x}\\
&+\delta^{-1}\big(K_2^2\|Q^{\frac12}\|^2_{HS(U,H^{\gamma}(D))}+\delta^{-2}\|Q^{\frac12}\|^2_{HS(U,H^{1+\gamma}(D))}\big)\int_{D}(1+|{\bf E}|^2+|{\bf H}|^2){\rm d}{\bf x}\\
\leq & C\|Q^{\frac12}\|^{2}_{HS(U,H^{1+\gamma}(D))}(1+\|u\|^2_{{\mathcal D}(M)}),
%\end{split}
\end{align*}
where we have used the Sobolev embedding $H^{\gamma}(D)\subset L^{\infty}(D)$ for any $\gamma> d/2$.

\subsection{Uniform boundedness of the solution}

We are now ready to establish establish the uniform boundedness of the solution   of stochastic Maxwell equations \eqref{sM_equations} in $L^{p}(\Omega;{\mathcal D}(M^k))$-norm, which is stated in the following theorem.
\begin{proposition}\label{regularity}
	Let Assumptions \ref{assump0}-\ref{assump 2} be fulfilled, and suppose that $u_0$ be an ${\mathcal F}_0$-measurable ${\mathbb H}$-valued random variable satisfying $\|u_0\|_{L^{p}(\Omega;{\mathcal D}(M^k))}<\infty$ for some $p\geq 2$. Then the mild solution \eqref{mild sol} satisfies
	\begin{eqnarray}
	\sup_{t\in[0,T]}\|u(t)\|_{L^{p}(\Omega;{\mathcal D}(M^k))}\leq C(1+\|u_0\|_{L^p(\Omega;{{\mathcal D}(M^k)})}),
	\end{eqnarray}
	where the positive constant $C$ may depend on $p$, $T$, and $\|Q^{\frac12}\|_{HS(U,H^{k+\gamma}(D))}$.
\end{proposition}
\begin{proof}
		Under Assumptions \ref{assump0}-\ref{assump 2}, we see that from \eqref{linear growth of F Ml} and \eqref{linear growth of B Ml}
	$F$ and $B$ are linear growth functions. Using the Burkholder-Davis-Gundy-type inequality for stochastic integrals \cite[Theorem 4.36]{PZ2014},  we have
	 for the mild solution \eqref{mild sol},
	\begin{equation*}
	\begin{split}
	{\mathbb E}\|u(t)\|^{p}_{{\mathcal D}(M^k)}	\lesssim& {\mathbb E}\|S(t)u_0\|^{p}_{{\mathcal D}(M^k)}
	+{\mathbb E}\int_0^t\|S(t-s)F(s,u(s))\|^{p}_{{\mathcal D}(M^k)}{\rm d}s\\
	&	+{\mathbb E}\Big\|\int_0^t S(t-s)B(s,u(s)){\rm d}W(s)\Big\|^{p}_{{\mathcal D}(M^k)}\\
	\lesssim& {\mathbb E}\|u_0\|^{p}_{{\mathcal D}(M^k)}+\int_0^t (1+{\mathbb E}\|u(s)\|^{p}_{{\mathcal D}(M^k)}){\rm d}s
	+\Big[{\mathbb E}\int_0^t\|B(s,u(s))\|_{HS(U_0,{{\mathcal D}(M^k)})}^{2}{\rm d}s\Big]^{\frac{p}{2}}\\
	\lesssim & \|u_0\|^p_{L^p(\Omega;{{\mathcal D}(M^k)})}+\int_0^t (1+{\mathbb E}\|u(s)\|^{p}_{{\mathcal D}(M^k)}){\rm d}s\\
	&+\|Q^{\frac12}\|^{p}_{HS(U,H^{k+\gamma}(D))}\int_0^t (1+{\mathbb E}\|u(s)\|^{p}_{{\mathcal D}(M^k)}){\rm d}s.
	\end{split}
	\end{equation*}
	By Gronwall's inequality, there exists a positive constant $C:=C(p,T,\|Q^{\frac12}\|_{HS(U,H^{k+\gamma}(D))})$ such that
	\[
	{\mathbb E}\|u(t)\|_{{\mathcal D}(M^k)}^p\leq C(1+\|u_0\|^p_{L^p(\Omega;{{\mathcal D}(M^k)})}),~\forall ~t\in[0,T].
	\]
	Therefore, we complete the proof.
\end{proof}

\begin{corollary}\label{regularity2}
	Under the same conditions of Proposition \ref{regularity}, there exists a constant $C:=C(p,T,u_0,Q)$ such that
	\begin{align}\label{est sup_Mk}
	{\mathbb E}\Big(\sup_{t\in[0,T]}\|u(t)\|^{p}_{{\mathcal D}(M^k)}\Big)\leq C.
	\end{align}
\end{corollary}
\begin{proof}
	The main step to derive \eqref{est sup_Mk} from the mild solution \eqref{mild sol} is that we need to deal with the stochastic integral
	\[
	{\mathbb E}\left[\sup_{t\in[0,T]}\left\| \int_0^t S(t-s)B(s,u(s)){\rm d}W(s) \right\|^{p}_{{\mathcal D}(M^k)}\right].
	\]
	By using the Burkholder-Davis-Gundy-type inequality for stochastic convolution \cite[Proposition 7.3]{PZ2014}, we have
	\begin{equation*}
	\begin{split}
	{\mathbb E}\left[\sup_{t\in[0,T]}\left\| \int_0^t S(t-s)B(s,u(s)){\rm d}W(s) \right\|^{p}_{{\mathcal D}(M^k)}\right]
	&\lesssim {\mathbb E}\int_{0}^{T}\|S(t-s)B(s,u(s))\|^{p}_{HS(U_0,{\mathcal D}(M^k))}{\rm d}s\\
	&	\lesssim \|Q^{\frac12}\|^{p}_{HS(U,H^{k+\gamma}(D))}\int_0^t (1+{\mathbb E}\|u(s)\|^{p}_{{\mathcal D}(M^k)}){\rm d}s\\
	&\leq C(p,T,u_0,Q),
	\end{split}
	\end{equation*}
	where we use the result of Proposition \ref{regularity} in the last step.
\end{proof}

\subsection{H\"older continuity of the solution}

In this subsection, we shall obtain the H\"older continuity of the solution of stochastic Maxwell equations \eqref{sM_equations} in $L^{2}(\Omega;{\mathcal D}(M^{k-1}))$-norm. To this end, we first give an very useful lemma.

\begin{lemma}\label{est ope}
	For the semigroup $\{S(t);t\geq 0\}$ on ${\mathbb H}$, it holds that
	\begin{equation}
	\|S(t)-Id\|_{{\mathcal L}({\mathcal D}(M);{\mathbb H})}\leq Ct,
	\end{equation}
	where the constant $C$ does not depend on $t$.
\end{lemma}
\begin{proof}
	We start from the system
	\begin{equation}\label{linear Max}
	\begin{cases}
	\frac{\partial u(t)}{\partial t}=Mu(t),~t\in(0,T],\\
	u(0)=u_0.
	\end{cases}
	\end{equation}
	Thus
	\[
	\frac{\partial}{\partial t}\|u(t)\|^2_{\mathbb H}=2\left\langle\frac{\partial u(t)}{\partial t},~u(t) \right\rangle_{\mathbb H}=2\left\langle Mu(t),~u(t) \right \rangle_{{\mathbb H}}=0,
	\]
	leads to
	\[
	\|u(t)\|_{\mathbb H}=	\|S(t)u_0\|_{\mathbb H}=\|u_0\|_{\mathbb H},
	\]
	which means $\|S(t)\|_{{\mathcal L}({\mathbb H};{\mathbb H})}=1$.
	
	Similarly, consider
	\[
	\frac{\partial}{\partial t}\|Mu(t)\|^2_{\mathbb H}=2\left\langle M\frac{\partial u(t)}{\partial t},~Mu(t) \right\rangle_{\mathbb H}=2\left\langle M^2u(t),~Mu(t) \right \rangle_{{\mathbb H}}=0,
	\]
	which leads to $\|S(t)\|_{{\mathcal L}({\mathcal D}(M);{\mathcal D}(M))}=1$.
	
	The assertion in this lemma is equivalent to
	$$\|u(t)-u_0\|_{\mathbb H}=\|(S(t)-Id)u_0\|_{\mathbb H}\leq C\|u_0\|_{{\mathcal D}(M)}t.$$
	In fact, we can conclude from \eqref{linear Max} that
	\begin{equation*}
	\begin{split}
	\langle u(t)-u_0,~u(t)\rangle_{{\mathbb H}}=\left\langle  \int_0^t Mu(s){\rm d}s,~u(t) \right\rangle_{{\mathbb H}},
	\end{split}
	\end{equation*}
	where the term in left-hand side is
	\[
	\frac12\Big(\|u(t)\|^2_{\mathbb H}-\|u_0\|^2_{\mathbb H}+\|u(t)-u_0\|^2_{\mathbb H}\Big)=\frac12\|u(t)-u_0\|^2_{\mathbb H},
	\]
	and the term in right-hand side can be estimated by
	\begin{equation*}
	\begin{split}
	\left\langle  \int_0^t Mu(s){\rm d}s,~u(t) \right\rangle_{{\mathbb H}}
	&\leq \int_{0}^{t}\|Mu(s)\|_{\mathbb H}\|u(t)\|_{\mathbb H}{\rm d}s\\
	&\leq \|u_0\|_{\mathbb H}\int_0^t\|u(s)\|_{{\mathcal D}(M)}{\rm d}s
	\leq C\|u_0\|^2_{{\mathcal D}(M)}t.
	\end{split}
	\end{equation*}
	Therefore we complete the proof.
\end{proof}

\begin{proposition}\label{holder}
	Under the same assumption of Proposition \ref{regularity}, we have for $0\leq t,s\leq T$,
	\begin{eqnarray}
	{\mathbb E}\|u(t)-u(s)\|_{{\mathcal D}(M^{k-1})}^{p}\leq C|t-s|^{p/2},\\
	\|{\mathbb E}(u(t)-u(s)\|_{{\mathcal D}(M^{k-1})}\leq C|t-s|,
	\end{eqnarray}
	where the positive constant $C$ may depend on $p$, $T$,  $\|Q^{\frac12}\|_{HS(U,H^{k+\gamma}(D))}$, and $\|u_0\|_{L^p(\Omega;{\mathcal D}(M^k))}$.
\end{proposition}
\begin{proof}
	From equation \eqref{mild sol}, we have
	\begin{equation}\label{u(t)-u(s)_eq}
	\begin{split}
	u(t)-u(s)=&\left(S(t-s)-I\right)u(s)+\int_{s}^{t}S(t-r)F(r,u(r)){\rm d}r\\[1.5mm]
	&	+\int_{s}^{t}S(t-r)B(r,u(r)){\rm d}W(r).
	\end{split}
	\end{equation}
	Therefore,
	\begin{equation}
	\begin{split}
	\mathbb{E}&\left\|u(t)-u(s)\right\|_{{\mathcal D}(M^{k-1})}^p\lesssim\mathbb{E}\left\|(S(t-s)-I)u(s)\right\|_{{\mathcal D}(M^{k-1})}^p\\[1.5mm]
	&+\mathbb{E}\left\|\int_{s}^{t}S(t-r)F(r,u(r)){\rm d}r\right\|_{{\mathcal D}(M^{k-1})}^p
	+\mathbb{E}\left\|\int_{s}^{t}S(t-r)B(r,u(r)){\rm d}W(r)\right\|_{{\mathcal D}(M^{k-1})}^p\\[1.5mm]
	&:=I+II+III.
	\end{split}
	\end{equation}
	
	For the first term $I$, we have
	%	{\color{red}
	\begin{equation*}
	\begin{split}
	I=\mathbb{E}\left\|(S(t)-S(s))u(s)\right\|_{{\mathcal D}(M^{k-1})}^p&\leq\left\|S(t)-S(s)\right\|_{{\mathcal L}({\mathcal D}(M^{k}),{\mathcal D}(M^{k-1}))}^p\mathbb{E}\left\|u(s)\right\|_{{\mathcal D}(M^{k})}^p\\[1mm]
	&\leq C(t-s)^p\|u(s)\|_{L^p(\Omega;{\mathcal D}(M^{k}))}^p,
	\end{split}
	\end{equation*}
	where we use the estimate $\|S(t)-I\|_{{\mathcal L}({\mathcal D}(M),{\mathbb H})}\leq Ct$ (see Lemma \ref{est ope}) in the last step.
	From Proposition \ref{regularity}, we have
	\begin{equation}\label{1I}
	I\leq C(1+\|u_0\|^p_{L^p(\Omega;{\mathcal D}(M^k))}) (t-s)^p.
	\end{equation}
	
	%	}
	For the second term $II$, it holds
	\begin{equation}\label{2II}
	\begin{split}
	II&=\mathbb{E}\left\|\int_{s}^{t}S(t-r)F(r,u(r)){\rm d}r\right\|_{{\mathcal D}(M^{k-1})}^p\\
	&\lesssim(t-s)^{p-1}\int_{s}^{t}\mathbb{E}\left\|S(t-r)F(r,u(r))\right\|_{{\mathcal D}(M^{k-1})}^p{\rm d}r\\[1mm]
	&\leq(t-s)^{p-1}\int_{s}^{t}\mathbb{E}\|S(t-r)\|^{p}_{{\mathcal L}({\mathcal D}(M^{k-1}),{\mathcal D}(M^{k-1}))}\|F(r,u(r))\|_{{\mathcal D}(M^{k-1})}^p{\rm d}r\\[1mm]
	&=(t-s)^{p-1}\int_{s}^{t}\mathbb{E}\|F(r,u(r))\|_{{\mathcal D}(M^{k-1})}^p{\rm d}r\\
	&	\leq (t-s)^{p-1}\int_s^t {\mathbb E}(1+\|u(r)\|_{{\mathcal D}(M^{k-1})}^p){\rm d}r
	\leq C(t-s)^p,
	\end{split}
	\end{equation}
	where in the last step, we utilize the estimate $\sup_{t\in[0,T]}{\mathbb E}\|u(t)\|_{{\mathcal D}(M^{k-1})}^p\leq C(1+\|u_0\|_{L^p(\Omega;{\mathcal D}(M^{k-1}))}^p)$ with the constant $C:=C(p,T,\|Q^{\frac12}\|_{HS(U,H^{k-1+\gamma}(D))}).$
	
	Using the Burkholder-Davis-Gundy-type inequality for stochastic integrals \cite[Theorem 4.36]{PZ2014}, we obtain,
	\begin{equation}\label{3III}
	\begin{split}
	III&\lesssim\Big(\int_s^t{\mathbb E}\left\|S(t-r)B(r,u(r))\right\|_{HS(U_0,{\mathcal D}(M^{k-1}))}^2{\rm d}r\Big)^{p/2}\\
	&	\leq \Big(\int_s^t \|S(t-r)\|^2_{{\mathcal L}({\mathcal D}(M^{k-1}),{\mathcal D}(M^{k-1}))}{\mathbb E}\|B(r,u(r))\|^2_{HS(U_0,{\mathcal D}(M^{k-1})}{\rm d}r\Big)^{p/2}\\
	&\leq\|Q^{\frac12}\|^{p}_{HS(U,H^{k-1+\gamma}(D))} \Big(\int_s^t{\mathbb E}(1+\|u(r)\|_{{\mathcal D}(M^{k-1})}^2){\rm d}r\Big)^{p/2}
	\leq C(t-s)^{p/2}.
	\end{split}
	\end{equation}
	Combining equations \eqref{1I}, \eqref{2II} and \eqref{3III} and based on the assumption $u_0\in {\mathcal D}(M^k)$, we obtain the first result.
	
	To get the second assertion, we take the expectation to the both sides of equation \eqref{u(t)-u(s)_eq}, it yields
	\begin{equation}
	\begin{split}
	\mathbb{E}(u(t)-u(s))=\mathbb{E}\left(\left(S(t-s)-I\right)u(s)\right)+\mathbb{E}\left(\int_{s}^{t}S(t-r)F(r,u(r)){\rm d}r\right)
	\end{split}
	\end{equation}
	Therefore, similar as \eqref{1I} and \eqref{2II} we get
	\begin{equation}
	\begin{split}
	\|\mathbb{E}(u(t)-u(s))\|_{{\mathcal D}(M^{k-1})}&\leq\left\|\mathbb{E}\left((S(t-s)-I)u(s)\right)\right\|_{{\mathcal D}(M^{k-1})}+\int_{s}^{t}\left\|\mathbb{E}(S(t-r)F(r,u(r)))\right\|_{{\mathcal D}(M^{k-1})}{\rm d}r\\
	&	\leq\mathbb{E}\left\|\left((S(t-s)-I)u(s)\right)\right\|_{{\mathcal D}(M^{k-1})}+\mathbb{E}\int_{s}^{t}\left\|(S(t-r)F(r,u(r)))\right\|_{{\mathcal D}(M^{k-1})}{\rm d}r\\
	&\leq C(t-s).
	\end{split}
	\end{equation}
	Therefore we finish the proof.
\end{proof}

\section{Temporal semidiscretization}
In this section, we apply semi-implicit Euler scheme to discretize stochastic Maxwell equations \eqref{sM_equations} in temporal direction, and investigate the convergence order in mean-square sense of this scheme. For the time interval $[0,~T]$, we introduce a uniform partition with step size $\tau=\frac{T}{N}$:
\begin{equation}
0=t_0<t_1<\ldots<t_{N}=T.
\end{equation}

Applying the semi-implicit Euler scheme to equation \eqref{sM_equations} in temporal direction, we have
\begin{equation}\label{2.20}
\begin{split}
&u^{n+1}=u^{n}+\tau Mu^{n+1}+\tau F(t_{n+1},u^{n+1})+B(t_n,u^{n})\Delta W^{n+1},\\[1mm]
&u^{0}=u_0,
\end{split}
\end{equation}
where  the increment $\Delta W^{n+1}$ is given by
\begin{equation*}
\Delta W^{n+1}:=W(t_{n+1})-W(t_{n})=\sum_{j=1}^{\infty}(\beta_j(t_{n+1})-\beta_{j}(t_{n}))Q^{\frac12}e_j.
\end{equation*}
Recall $u^n=\begin{pmatrix}
{\bf E}^n\\{\bf H}^n
\end{pmatrix}$, then scheme \eqref{2.20} is equivalent to
\begin{equation}\label{Euler scheme}
\begin{split}
&\varepsilon{\bf E}^{n+1}=\varepsilon{\bf E}^n+\tau\nabla\times{\bf H}^{n+1}
-\tau{\bf J}_e(t_{n+1},{\bf E}^{n+1},{\bf H}^{n+1})-{\bf J}_e^r(t_n,{\bf E}^n,{\bf H}^n)\Delta W^{n+1},\\
&\mu{\bf H}^{n+1}=\mu{\bf H}^n-\tau\nabla\times{\bf E}^{n+1}-\tau{\bf J}_m(t_{n+1},{\bf E}^{n+1},{\bf H}^{n+1})-{\bf J}_m^r(t_n,{\bf E}^n,{\bf H}^n)\Delta W^{n+1},\\
&{\bf E}^0={\bf E}_0,~{\bf H}^0={\bf H}_0.
\end{split}
\end{equation}

\subsection{Properties of the discrete solution}
In this subsection, we will show that  there exists a ${\mathcal D}(M)$-valued $\{{\mathcal F}_{t_n}\}_{0\leq n\leq N}$-adapted discrete solution $\{ u^n; ~n=0,1,\ldots,N\}$ for scheme \eqref{2.20} or $\{ {\bf E}^n,{\bf H}^n;~n=0,1,\ldots,N\}$ for scheme \eqref{Euler scheme}.

\begin{lemma}
	For a fixed $T=t_{N}>0$, let
	$p\geq 2$ and   $\tau\leq \tau^{*}$ with $\tau^{*}:= \tau^{*}(\|u_0\|_{L^p(\Omega;{\mathcal D}(M))},T,p)$. There exists a ${\mathcal D}(M)$-valued $\{{\mathcal F}_{t_n}\}_{0\leq n\leq N}$-adapted discrete solution $\{ u^n; ~n=0,1,\ldots,N\}$ of the scheme \eqref{2.20}, and a constant $C:=C(p,T,\|Q^{\frac12}\|_{HS(U,H^{1+\gamma}(D))})>0$ such that
	\begin{equation}\label{bound un}
	\max_{1\leq n\leq N}{\mathbb E}\|u^n\|^{p}_{{\mathcal D}(M)}\leq C(1+\|u_0\|^{p}_{L^{p}(\Omega;{\mathcal D}(M))}).
	\end{equation}
\end{lemma}
\begin{proof}
	{\em Step 1: Existence and $\{{\mathcal F}_{t_n}\}_{0\leq n\leq N}$-adaptedness.} Fix a set $\Omega^{'}\subset\Omega$, ${\mathbb P}(\Omega^{'})=1$ such that $W(t,\omega)\in U$ for all $t\in[0,T]$ and $\omega\in\Omega^{'}$. In the following, let us assume that $\omega\in\Omega^{'}$. The existence of iterates $\{ u^n; ~n=0,1,\ldots,N\}$
	follows from a standard Galerkin method and Brouwer's theorem, in combining with assertion \eqref{bound un}.
	
	Define a map
	\begin{equation*}
	\begin{split}
	\Lambda:~&{\mathcal D}(M)\times U\to{\mathcal P}({\mathcal D}(M))\\
	&(u^n,\Delta W^{n+1})\to \Lambda(u^n,\Delta W^{n+1})
	\end{split}
	\end{equation*}
	where ${\mathcal P}({\mathcal D}(M))$ denotes the set of all subsets of ${\mathcal D}(M)$, and $\Lambda(u^n,\Delta W^{n+1})$ is the set of solutions $u^{n+1}$ of \eqref{2.20}. By the closedness of the graph of $\Lambda$ and a selector theorem \cite[Theorem 3.1]{BT1973}, there exists a universally and Borel measurable mapping $\lambda_n:~{\mathcal D}(M)\times U\to {\mathcal D}(M)$ such that $\lambda_n(s_1,s_2)\in\Lambda(s_1,s_2)$ for all $(s_1,s_2)\in{\mathcal D}(M)\times U$. Therefore, ${\mathcal F}_{t_{n+1}}$-measurability of $u^{n+1}$ follows from the Doob-Dynkin lemma.
	
	{\em Step 2: Case $p=2$ for \eqref{bound un}.}
	We  apply $\langle\cdot,~u^{n+1}\rangle_{{\mathbb H}}$ into both sides of \eqref{2.20} and get
	\begin{equation}\label{4.5}
	\begin{split}
	&	\frac12\Big(\|u^{n+1}\|^2_{\mathbb H}-\|u^n\|^2_{\mathbb H}\Big)+\frac12\|u^{n+1}-u^{n}\|^2_{\mathbb H}\\
	&	=\tau\langle F(t_{n+1},u^{n+1}),~u^{n+1}\rangle_{{\mathbb H}}
	+\langle B(t_n,u^n)\Delta W^{n+1},~u^{n+1}\rangle_{{\mathbb H}}\\
	&\leq C\tau(1+\|u^{n+1}\|_{\mathbb H})\|u^{n+1}\|_{\mathbb H}
	+\|B(t_n,u^n)\Delta W^{n+1}\|_{\mathbb H}^2\\
	&\quad+\frac14\|u^{n+1}-u^n\|_{\mathbb H}^2+\langle B(t_n,u^n)\Delta W^{n+1},~u^{n}\rangle_{{\mathbb H}},
	\end{split}
	\end{equation}
	which gives
	\begin{equation}\label{4.6}
	\begin{split}
	&	\frac12\Big(\|u^{n+1}\|^2_{\mathbb H}-\|u^n\|^2_{\mathbb H}\Big)+\frac14\|u^{n+1}-u^{n}\|^2_{\mathbb H}\\
	&	\leq C\tau +C\tau\|u^{n+1}\|_{\mathbb H}^2+\|B(t_n,u^n)\Delta W^{n+1}\|_{\mathbb H}^2
	+\langle B(t_n,u^n)\Delta W^{n+1},~u^{n}\rangle_{{\mathbb H}}.
	\end{split}
	\end{equation}
	Next we apply $\langle\cdot,~Mu^{n+1}-Mu^{n}\rangle_{{\mathbb H}}$ into both sides of \eqref{2.20} and get
	\begin{equation}\label{4.7}
	\begin{split}
	&	\frac{1}{2}\Big(\|Mu^{n+1}\|^2_{\mathbb H}-\|Mu^{n}\|^2_{\mathbb H}\Big)
	+\frac12\|Mu^{n+1}-Mu^{n}\|^2_{\mathbb H}\\
	&	=-\langle F(t_{n+1},u^{n+1}),~Mu^{n+1}-Mu^{n}\rangle_{{\mathbb H}}
	-\frac{1}{\tau}\langle B(t_n,u^n)\Delta W^{n+1},~Mu^{n+1}-Mu^{n}\rangle_{{\mathbb H}}\\
	&:=I+II.
	\end{split}
	\end{equation}
	For the term $I$, using the skew adjoint property of operator $M$ and \eqref{2.20}, we get
	\begin{equation}\label{4.8}
	\begin{split}
	I=&\langle MF(t_{n+1},u^{n+1}),~u^{n+1}-u^{n}\rangle_{{\mathbb H}}\\
	=&\langle MF(t_{n+1},u^{n+1}),~\tau Mu^{n+1}+\tau F(t_{n+1},u^{n+1})+B(t_n,u^n)\Delta W^{n+1}\rangle_{{\mathbb H}}\\
	\leq &\tau\|MF(t_{n+1},u^{n+1})\|_{\mathbb H}\|Mu^{n+1}\|_{\mathbb H}
	+\langle MF(t_{n+1},u^{n+1}),~B(t_n,u^n)\Delta W^{n+1}\rangle_{\mathbb H}\\
	\leq & C\tau+C\tau\|Mu^{n+1}\|^2_{\mathbb H}
	+\langle MF(t_{n+1},u^{n+1}),~B(t_n,u^n)\Delta W^{n+1}\rangle_{\mathbb H}.
	\end{split}
	\end{equation}
	Similarly, for the term $II$, we get
	\begin{equation}\label{4.9}
	\begin{split}
	II=&\frac{1}{\tau}\langle M(B(t_n,u^n)\Delta W^{n+1}),~u^{n+1}-u^{n}\rangle_{\mathbb H}\\
	=&\frac{1}{\tau}\langle M(B(t_n,u^n)\Delta W^{n+1}),~\tau Mu^{n+1}+\tau F(t_{n+1},u^{n+1})+B(t_n,u^n)\Delta W^{n+1}\rangle_{\mathbb H}\\
	=&\langle M(B(t_n,u^n)\Delta W^{n+1}),~ Mu^{n+1}-Mu^{n}\rangle_{\mathbb H}
	+	\langle M(B(t_n,u^n)\Delta W^{n+1}),~ Mu^{n}\rangle_{\mathbb H}\\
	&	+\langle M(B(t_n,u^n)\Delta W^{n+1}),~F(t_{n+1},u^{n+1})\rangle_{\mathbb H}\\
	&\leq \frac14\|Mu^{n+1}-Mu^{n}\|^2_{\mathbb H}+\|M(B(t_n,u^n)\Delta W^{n+1})\|^2_{\mathbb H}+\langle M(B(t_n,u^n)\Delta W^{n+1}),~ Mu^{n}\rangle_{\mathbb H}\\
	&-\langle B(t_n,u^n)\Delta W^{n+1},~MF(t_{n+1},u^{n+1})\rangle_{\mathbb H}.
	\end{split}
	\end{equation}
	Substituting \eqref{4.8} and \eqref{4.9} into \eqref{4.7}, we have
	\begin{equation}\label{4.10}
	\begin{split}
	&	\frac{1}{2}\Big(\|Mu^{n+1}\|^2_{\mathbb H}-\|Mu^{n}\|^2_{\mathbb H}\Big)
	+\frac14\|Mu^{n+1}-Mu^{n}\|^2_{\mathbb H}\\
	&\leq	C\tau+C\tau\|Mu^{n+1}\|^2_{\mathbb H}
	+\|M(B(t_n,u^n)\Delta W^{n+1})\|^2_{\mathbb H}
	+\langle M(B(t_n,u^n)\Delta W^{n+1}),~ Mu^{n}\rangle_{\mathbb H}.
	\end{split}
	\end{equation}
	Summing \eqref{4.6} and \eqref{4.10} together leads to
	\begin{equation}\label{4.11}
	\begin{split}
	&\frac12\Big(\|u^{n+1}\|^2_{{\mathcal D}(M)}-\|u^n\|^2_{{\mathcal D}(M)}\Big)
	+\frac14\|u^{n+1}-u^{n}\|_{{\mathcal D}(M)}^2\\
	&\leq C\tau+C\tau\|u^{n+1}\|^2_{{\mathcal D}(M)}+\|B(t_n,u^n)\Delta W^{n+1}\|^2_{{\mathcal D}(M)}+\langle B(t_n,u^n)\Delta W^{n+1},~u^{n}\rangle_{{\mathcal D}(M)}.
	\end{split}
	\end{equation}
	After applying expectation on both sides of \eqref{4.11}, one arrives at
	\begin{equation*}
	\begin{split}
	&	\frac12\Big({\mathbb E}\|u^{n+1}\|^2_{{\mathcal D}(M)}-{\mathbb E}\|u^n\|^2_{{\mathcal D}(M)}\Big)
	+\frac14{\mathbb E}\|u^{n+1}-u^{n}\|_{{\mathcal D}(M)}^2\\
	&	\leq C\tau+C\tau{\mathbb E}\|u^{n+1}\|^2_{{\mathcal D}(M)}
	+C\|Q^{\frac12}\|^2_{HS(U,H^{1+\gamma}(D))}\tau(1+{\mathbb E}\|u^n\|^2_{{\mathcal D}(M)}).
	\end{split}
	\end{equation*}
	The discrete Gronwall's lemma then leads to the assertion of this lemma in case $\tau\leq \tau^{*}$ is chosen.
	
	{\em Step 3: Case $p>2$ for \eqref{bound un}.}
	In order to show assertion \eqref{bound un}, we employ an inductive argument.
	To obtain the result for $p=4$, we multiply \eqref{4.11} by $\|u^{n+1}\|^2_{{\mathcal D}(M)}$, and use the identity $(a-b)a=\frac12\big(a^2-b^2+(a-b)^2\big)$ where $a,b\in{\mathbb R}$, to get
	\begin{equation}
	\begin{split}
	&	\frac14\Big(\|u^{n+1}\|^4_{{\mathcal D}(M)}-\|u^{n}\|^4_{{\mathcal D}(M)}\Big)+\frac14(\|u^{n+1}\|^2_{{\mathcal D}(M)}-\|u^{n}\|^2_{{\mathcal D}(M)})^2+\frac14\|u^{n+1}-u^{n}\|_{{\mathcal D}(M)}^2\|u^{n+1}\|^2_{{\mathcal D}(M)}\\
	&	\leq C\tau\|u^{n+1}\|^2_{{\mathcal D}(M)}+C\tau\|u^{n+1}\|^4_{{\mathcal D}(M)}+\|B(t_n,u^n)\Delta W^{n+1}\|^2_{{\mathcal D}(M)}\|u^{n+1}\|^2_{{\mathcal D}(M)}\\
	&\quad+\langle B(t_n,u^n)\Delta W^{n+1},~u^{n}\rangle_{{\mathcal D}(M)}\|u^{n+1}\|^2_{{\mathcal D}(M)}\\
	&\leq C\tau\|u^{n+1}\|^2_{{\mathcal D}(M)}+C\tau\|u^{n+1}\|^4_{{\mathcal D}(M)}+\frac{1}{\tau}\|B(t_n,u^n)\Delta W^{n+1}\|^4_{{\mathcal D}(M)}
	+\frac18(\|u^{n+1}\|^2_{{\mathcal D}(M)}-\|u^{n}\|^2_{{\mathcal D}(M)})^2\\
	&\quad+(\langle B(t_n,u^n)\Delta W^{n+1},~u^{n}\rangle_{{\mathcal D}(M)})^2+
	\langle B(t_n,u^n)\Delta W^{n+1},~u^{n}\rangle_{{\mathcal D}(M)}\|u^{n}\|^2_{{\mathcal D}(M)}.
	\end{split}
	\end{equation}
	After applying expectation on both sides of the above inequality and using the linear growth property of $B$, one gets
	\begin{equation}
	\begin{split}
	&	\frac14\Big({\mathbb E}\|u^{n+1}\|^4_{{\mathcal D}(M)}-{\mathbb E}\|u^{n}\|^4_{{\mathcal D}(M)}\Big)+\frac18{\mathbb E}(\|u^{n+1}\|^2_{{\mathcal D}(M)}-\|u^{n}\|^2_{{\mathcal D}(M)})^2\\
	&\leq C\tau{\mathbb E}	\|u^{n+1}\|^2_{{\mathcal D}(M)}+C\tau{\mathbb E}\|u^{n+1}\|^4_{{\mathcal D}(M)}+C\|Q^{\frac12}\|_{HS(U,H^{1+\gamma}(D))}^4\tau(1+{\mathbb E}\|u^{n}\|^4_{{\mathcal D}(M)})\\
	&\quad+C\|Q^{\frac12}\|_{HS(U,H^{1+\gamma}(D))}^2{\mathbb E}\big(\|u^{n}\|^2_{{\mathcal D}(M)}(1+\|u^{n}\|^2_{{\mathcal D}(M)})\big)
	\end{split}
	\end{equation}
	The discrete Gronwall's lemma then leads to the assertion for $p=4$ in case $\tau\leq \tau^{*}$ is chosen.
	
	Using the case when $p=2$ and $p=4$, it is easy to check that the following holds true
	\[
	{\mathbb E}\|u^{n}\|^3_{{\mathcal D}(M)}\leq \frac12\Big({\mathbb E}\|u^{n}\|^2_{{\mathcal D}(M)}+{\mathbb E}\|u^{n}\|^4_{{\mathcal D}(M)}\Big)
	\leq C,
	\]
	which leads to the assertion for $p=3$.
	
	By repeating the above procedure, we could show that the assertion holds for general $p\geq 2$. Thus we complete the proof.
\end{proof}

\subsection{Mean-square convergence order}
In this subsection we investigate the convergence order in mean-square sense of semidiscretization \eqref{2.20} via the truncation error approach.

Denote by $\delta^{n+1}$ the truncation error of the semi-implicit Euler scheme, i.e., 	
\begin{equation}\label{delta_n+1}
\delta^{n+1}:=u(t_{n+1})-u(t_n)-\tau Mu(t_{n+1})-\tau F(t_{n+1},u(t_{n+1}))-B(t_n,u(t_n))\Delta W^{n+1},
\end{equation}	
where $u(t)$ means the solution of stochastic Maxwell equations \eqref{sM_equations}. The estimate of this truncation error is stated in the following lemma.

\begin{lemma}\label{estimate_delta_n+1}
	Let Assumptions \ref{assump0}-\ref{assump 2} be fulfilled with $k=2$, and suppose that $u_0$ is an ${\mathcal F}_0$-measurable random variable satisfying $\|u_0\|_{L^2(\Omega;{\mathcal D}(M^2))}<\infty$. Then we have
	\begin{equation}
	\begin{split}
	\mathbb{E}\|\delta^{n+1}\|_{\mathbb H}^2\leq C\tau^2,\quad{\mathbb E}\|\mathbb{E}\big(\delta^{n+1}|{\mathcal F}_{t_{n}}\big)\|^2_{\mathbb H}\leq C\tau^{3},
	\end{split}
	\end{equation}
	where the positive constant $C$ depends on the Lipschitz coefficients of $F$ and $B$, $T$, $\|u_0\|_{L^2(\Omega;{\mathcal D}(M^2))}$ and $\|Q^{\frac12}\|_{HS(U,H^{2+\gamma}(D))}$.
\end{lemma}
\begin{proof}
	By replacing the expression
	\[u(t_{n+1})-u(t_n)
	=\int_{t_n}^{t_{n+1}}Mu(s){\rm d}s+\int_{t_n}^{t_{n+1}}F(s,u(s)){\rm d}s+\int_{t_n}^{t_{n+1}}B(s,u(s)){\rm d}W(s)\]
	into \eqref{delta_n+1}, we have
	\begin{equation}
	\begin{split}
	\delta^{n+1}=&\int_{t_n}^{t_{n+1}}\big(Mu(s)-Mu(t_{n+1}\big){\rm d}s+\int_{t_n}^{t_{n+1}}\big(F(s,u(s))-F(t_{n+1},u(t_{n+1}))\big){\rm d}s\\[1.5mm]
	&+\int_{t_n}^{t_{n+1}}\big(B(s,u(s))-B(t_n,u(t_n))\big){\rm d}W(s).
	\end{split}
	\end{equation}
	Then,
	\begin{equation*}
	\begin{split}
	\mathbb{E}\|\delta^{n+1}\|_{\mathbb H}^2\lesssim&\mathbb{E}\left\|\int_{t_n}^{t_{n+1}}M(u(s)-u(t_{n+1})){\rm d}s\right\|_{\mathbb H}^2+\mathbb{E}\left\|\int_{t_n}^{t_{n+1}}(F(s,u(s))-F(t_{n+1},u(t_{n+1}))){\rm d}s\right\|_{\mathbb H}^2\\[1.5mm]
	&+\mathbb{E}\left\|\int_{t_n}^{t_{n+1}}(B(s,u(s))-B(t_n,u(t_n))){\rm d}W(s)\right\|_{\mathbb H}^2\\[1.5mm]
	&=:I+II+III.
	\end{split}
	\end{equation*}
	Using H\"{o}lder inequality to the first term $I$ leads to
	\begin{equation}
	\begin{split}
	I\leq& \tau\mathbb{E}\int_{t_n}^{t_{n+1}}\left\|M(u(s)-u(t_{n+1}))\right\|_{\mathbb H}^2{\rm d}s\\[1mm]
	&\leq \tau\int_{t_n}^{t_{n+1}}\mathbb{E}\left\|u(s)-u(t_{n+1})\right\|_{{\mathcal D}(M)}^2{\rm d}s.
	\end{split}
	\end{equation}
	Based on Proposition \ref{holder}, it holds
	\begin{equation}\label{tm_I}
	\begin{split}
	I\leq C\tau^3,
	\end{split}
	\end{equation}
	where $C$ depends on the coefficients $F$ and $B$, $T$, $\|Q^{\frac12}\|_{HS(U,H^{2+\gamma}(D))}$ and $\|u_0\|_{L^2(\Omega;{\mathcal D}(M^2))}$.
	
	For the second term $II$, similarly, by Proposition \ref{holder} and continuous differentiability of $F$ with respect to $t$, we have
	\begin{equation}
	\begin{split}
	II\leq&\tau\mathbb{E}\int_{t_n}^{t_{n+1}}\left\|F(s,u(s))-F(t_{n+1},u(t_{n+1}))\right\|_{\mathbb H}^2{\rm d}s\\[1.5mm]
	\leq& \tau\mathbb{E}\int_{t_n}^{t_{n+1}}\left\|F(s,u(s))-F(s,u(t_{n+1}))\right\|_{\mathbb H}^2{\rm d}s\\[1.5mm]
	&+\tau\mathbb{E}\int_{t_n}^{t_{n+1}}\left\|F(t_{n+1},u(t_{n+1}))-F(s,u(t_{n+1}))\right\|_{\mathbb H}^2{\rm d}s\\[1.5mm]
	\leq& C\tau\int_{t_n}^{t_{n+1}}{\mathbb E}\|u(s)-u(t_{n+1})\|_{\mathbb H}^2+\|\partial_{t}F(\theta,u(t_{n+1}))(t_{n+1}-s)\|^2_{{\mathbb H}}{\rm d}s\\
	\leq& C\tau^3,
	\end{split}
	\end{equation}
		where $C$ depends on the coefficients $F$ and $B$, $T$, $\|Q^{\frac12}\|_{HS(U,H^{1+\gamma}(D))}$ and $\|u_0\|_{L^2(\Omega;{\mathcal D}(M))}$.
		
	By the infinite dimensional It\^{o} isometry formula, for the third term $III$ we get,
\begin{align*}
	III=&\mathbb{E}\int_{t_n}^{t_{n+1}}\left\|(B(s,u(s))-B(t_n,u(t_n)))\right\|_{HS(U_0,{\mathbb H})}^2{\rm d}s\\[1.5mm]
\leq& {\mathbb E}    \int_{t_n}^{t_{n+1}}\left\|(B(s,u(s))-B(s,u(t_n)))\right\|_{HS(U_0,{\mathbb H})}^2{\rm d}s\\[1.5mm]
&+{\mathbb E}   \int_{t_n}^{t_{n+1}}\left\|(B(s,u(t_n))-B(t_n,u(t_n)))\right\|_{HS(U_0,{\mathbb H})}^2{\rm d}s\\[1.5mm]
\leq &C\|Q^{\frac12}\|^2_{HS(U,H^{\gamma}(D))}{\mathbb E} \int_{t_n}^{t_{n+1}}\|u(s)-u(t_n)\|_{\mathbb H}^2{\rm d}s \\[1.5mm]
& +  {\mathbb E}    \int_{t_n}^{t_{n+1}}\left\|\partial_{t}B(\theta_1,u(t_n))(s-t_n)\right\|_{HS(U_0,{\mathbb H})}^2{\rm d}s\\[1.5mm]
\leq& C\tau^2,
\end{align*}
	where $C$ depends on the coefficients $F$ and $B$, $T$, $\|Q^{\frac12}\|_{HS(U,H^{1+\gamma}(D))}$ and $\|u_0\|_{L^2(\Omega;{\mathcal D}(M))}$.

	Combining the above equations, we can obtain the first assertion.
	
	In the similar way, we can prove that
	\begin{equation*}
	\begin{split}
	&{\mathbb E}\|\mathbb{E}\big(\delta^{n+1}|{\mathcal F}_{t_{n}}\big)\|^2_{\mathbb H}\\
	&\lesssim\left\|\mathbb{E}\left(\int_{t_n}^{t_{n+1}}M(u(s)-u(t_{n+1})){\rm d}s\right)\right\|^2_{\mathbb H}+\left\|\mathbb{E}\left(\int_{t_n}^{t_{n+1}}(F(s,u(s))-F(t_{n+1},u(t_{n+1}))){\rm d}s\right)\right\|^2_{\mathbb H}\\[1.5mm]
	&\leq\mathbb{E}\left\|\left(\int_{t_n}^{t_{n+1}}M(u(s)-u(t_{n+1})){\rm d}s\right)\right\|^2_{\mathbb H}+\mathbb{E}\left\|\left(\int_{t_n}^{t_{n+1}}(F(s,u(s))-F(t_{n+1},u(t_{n+1}))){\rm d}s\right)\right\|^2_{\mathbb H}\\[1.5mm]
	&\leq I+II\leq C\tau^{3},
	\end{split}
	\end{equation*}
		where $C$ depends on the coefficients $F$ and $B$, $T$, $\|Q^{\frac12}\|_{HS(U,H^{2+\gamma}(D))}$ and $\|u_0\|_{L^2(\Omega;{\mathcal D}(M^2))}$.
	Thus, we have finished the proof.
\end{proof}						

Denote by $e^n=u(t_n)-u^n$, then the main result of this paper is stated in the following theorem.
\begin{theorem}\label{thm:order}
	Under the same assumption of Lemma \ref{estimate_delta_n+1}, we have
	\begin{equation}
	\max_{0\leq n\leq N} \left(\mathbb{E}\|e^n\|_{\mathbb H}^2\right)^{1/2}\leq C\tau^{\frac12},
	\end{equation}
	where the positive constant $C$ may depend on the Lipschitz coefficients of $F$ and $B$, $T$, $\|u_0\|_{L^2(\Omega;{\mathcal D}(M^2))}$ and $\|Q^{\frac12}\|_{HS(U,H^{2+\gamma}(D))}$, but independent of $\tau$ and $n$.
\end{theorem}
\begin{proof}
	Subtracting equation \eqref{2.20} from \eqref{delta_n+1}, it leads to
	\begin{equation}
	\begin{split}
	e^{n+1}-e^{n}=&\delta^{n+1}+\tau Me^{n+1}+\tau\Big(F(t_{n+1},u(t_{n+1}))-F(t_{n+1},u^{n+1})\Big)\\
	&+\Big(B(t_n,u(t_n))-B(t_{n},u^n)\Big)\Delta W^{n+1}.
	\end{split}
	\end{equation}
	Taking the ${\mathbb H}$-inner product of the above equality with $e^{n+1}$, we get
	\begin{equation}\label{4.23}
	\begin{split}
	\langle e^{n+1}-e^{n},e^{n+1}\rangle_{\mathbb H}=&\langle\delta^{n+1},e^{n+1}\rangle_{\mathbb H}+\tau \langle Me^{n+1},e^{n+1}\rangle_{\mathbb H}\\
	&+ \tau\langle F(t_{n+1},u(t_{n+1}))-F(t_{n+1},u^{n+1}),e^{n+1}\rangle_{\mathbb H}\\
	&+\left\langle \Big(B(t_n,u(t_n))-B(t_{n},u^n)\Big)\Delta W^{n+1}, e^{n+1}\right\rangle_{{\mathbb H}}.
	\end{split}
	\end{equation}

	Noticing that $e^{n+1}=\frac{1}{2}(e^{n+1}-e^n)+\frac{1}{2}(e^{n+1}+e^n)$, for the left-hand side of the above equation, we have
	\begin{equation}\label{4.24}
	\begin{split}
	\langle e^{n+1}-e^{n},e^{n+1}\rangle_{{\mathbb H}}=\frac{1}{2}\|e^{n+1}\|_{\mathbb H}^2-\frac{1}{2}\|e^n\|_{\mathbb H}^2+\frac{1}{2}\|e^{n+1}-e^n\|_{\mathbb H}^2.
	\end{split}
	\end{equation}
	
	For the first term in the right-hand side of \eqref{4.23},
	it follows from $ab\leq a^2+\frac{1}{4}b^2$ that
	\begin{equation}
	\begin{split}
	\langle\delta^{n+1},e^{n+1}\rangle_{\mathbb H}=&\langle\delta^{n+1},e^{n+1}-e^n\rangle_{\mathbb H}+\langle\delta^{n+1},e^{n}\rangle_{\mathbb H}\\[1mm]
	&\leq \|\delta^{n+1}\|_ {\mathbb H}\cdot\|e^{n+1}-e^n\|_{\mathbb H}+\langle\delta^{n+1},e^{n}\rangle_ {\mathbb H}\\[1mm]
	&\lesssim \|\delta^{n+1}\|_{\mathbb H}^2+\frac{1}{8}\|e^{n+1}-e^n\|_ {\mathbb H}^2+\langle\delta^{n+1},e^{n}\rangle_{\mathbb H}.
	\end{split}
	\end{equation}
	After applying expectation on both sides of the above inequality, we have
	\begin{equation}\label{4.26}
	\begin{split}
	{\mathbb E}\langle\delta^{n+1},e^{n+1}\rangle_{\mathbb H}
	&\leq {\mathbb E}\|\delta^{n+1}\|_{\mathbb H}^2+\frac{1}{8}{\mathbb E}\|e^{n+1}-e^n\|_ {\mathbb H}^2+{\mathbb E}\langle{\mathbb E}\big(\delta^{n+1}|{\mathcal F}_{t_n}\big),e^{n}\rangle_{\mathbb H}\\
	&\leq{\mathbb E}\|\delta^{n+1}\|_{\mathbb H}^2+\frac{1}{8}{\mathbb E}\|e^{n+1}-e^n\|_ {\mathbb H}^2+{\mathbb E}\Big(\|{\mathbb E}\big(\delta^{n+1}|{\mathcal F}_{t_n}\big)\|_{\mathbb H}\|e^{n}\|_{\mathbb H}\Big)\\
	&\leq {\mathbb E}\|\delta^{n+1}\|_{\mathbb H}^2+\frac{1}{8}{\mathbb E}\|e^{n+1}-e^n\|_ {\mathbb H}^2+\frac{1}{\tau}{\mathbb E}\|{\mathbb E}\big(\delta^{n+1}|{\mathcal F}_{t_n}\big)\|^2_{\mathbb H}+\tau{\mathbb E}\|e^{n}\|^2_{\mathbb H}\\
	&\leq C\tau^2 +\tau{\mathbb E}\|e^{n}\|^2_{\mathbb H}+\frac{1}{8}{\mathbb E}\|e^{n+1}-e^n\|_ {\mathbb H}^2,
	\end{split}
	\end{equation}
	where in the last step, we utilize the results on the estimates for truncation error $\delta^{n+1}$ in Lemma \ref{estimate_delta_n+1}.
	
	For the second term in the right-hand side of \eqref{4.23},
	utilizing the skew-adjointness of the Maxwell operator $M$, it holds
	\begin{equation}\label{4.27}
	\langle Me^{n+1},e^{n+1}\rangle_{\mathbb H}=0.
	\end{equation}
	
	For the third and forth terms in the right-hand side of \eqref{4.23}, utilizing the global
	Lipschitz properties of $F$ and $B$, respectively, we obtain
	\begin{equation}\label{4.28}
	\begin{split}
	&\tau\langle F(t_{n+1},u(t_{n+1}))-F(t_{n+1},u^{n+1}),e^{n+1}\rangle_{\mathbb H}\\
	&\leq \tau\|F(t_{n+1},u(t_{n+1}))-F(t_{n+1},u^{n+1})\|_{\mathbb H}\|e^{n+1}\|_{\mathbb H}\\
	&\leq C\tau\|e^{n+1}\|_{\mathbb H}^2,
	\end{split}
	\end{equation}
	and
	\begin{equation}\label{4.29}
	\begin{split}
	&\left\langle \Big(B(t_n,u(t_n))-B(t_{n},u^n)\Big)\Delta W^{n+1}, e^{n+1}\right\rangle_{{\mathbb H}}\\[1.5mm]
	&=\left\langle \Big(B(t_n,u(t_n))-B(t_{n},u^n)\Big)\Delta W^{n+1}, e^{n+1}-e^{n}\right\rangle_{{\mathbb H}}\\[1.5mm]
	&\quad+\left\langle \Big(B(t_n,u(t_n))-B(t_{n},u^n)\Big)\Delta W^{n+1}, e^{n}\right\rangle_{{\mathbb H}}\\[1.5mm]
	&\leq \left\|\Big(B(t_n,u(t_n))-B(t_{n},u^n)\Big)\Delta W^{n+1}\right\|_{\mathbb H}^2 +\frac18\|e^{n+1}-e^{n}\|^2_{\mathbb H}\\[1.5mm]
	&\quad+\left\langle \Big(B(t_n,u(t_n))-B(t_{n},u^n)\Big)\Delta W^{n+1}, e^{n}\right\rangle_{{\mathbb H}}.
	\end{split}
	\end{equation}
	After applying expectation on both sides of the above inequality \eqref{4.29} and using the global Lipschitz property of $B$, we get
	\begin{equation}\label{4.30}
	\begin{split}
	&{\mathbb E}\left\langle \Big(B(t_n,u(t_n))-B(t_{n},u^n)\Big)\Delta W^{n+1}, e^{n+1}\right\rangle_{{\mathbb H}}\\[1.5mm]
	&\leq {\mathbb E}\left\|\Big(B(t_n,u(t_n))-B(t_{n},u^n)\Big)\Delta W^{n+1}\right\|_{\mathbb H}^2 +\frac18{\mathbb E}\|e^{n+1}-e^{n}\|^2_{\mathbb H}\\[1.5mm]
	&\leq \|Q^{\frac12}\|_{HS(U,H^{\gamma}(D))}^2\tau{\mathbb E}\left\|e^{n}\right\|_{\mathbb H}^2 +\frac18{\mathbb E}\|e^{n+1}-e^{n}\|^2_{\mathbb H}.
	\end{split}
	\end{equation}
	
	Substituting \eqref{4.24}, \eqref{4.26}, \eqref{4.27}, \eqref{4.28} and \eqref{4.30} into \eqref{4.23} leads to
	\begin{equation}\label{4.31}
	\begin{split}
	&\frac{1}{2}{\mathbb E}\|e^{n+1}\|_{\mathbb H}^2-\frac{1}{2}{\mathbb E}\|e^n\|_{\mathbb H}^2+\frac{1}{2}{\mathbb E}\|e^{n+1}-e^n\|_{\mathbb H}^2\\
	&\leq
	C\tau^2 +\tau{\mathbb E}\|e^{n}\|^2_{\mathbb H}+\frac{1}{4}{\mathbb E}\|e^{n+1}-e^n\|_ {\mathbb H}^2+C\tau{\mathbb E}\|e^{n+1}\|_{\mathbb H}^2+\|Q^{\frac12}\|_{HS(U,H^{\gamma}(D))}^2\tau{\mathbb E}\left\|e^{n}\right\|_{\mathbb H}^2.
	\end{split}
	\end{equation}
	The discrete Gronwall's lemma leads to the assertion in case $\tau\leq \tau^{*}$ is chosen.
	
	Thus, the proof is completed.
\end{proof}
\begin{remark}
	If a $\theta$-method is applied to discretize stochastic Maxwell equations \eqref{sM_equations} in the temporal direction, i.e.,
	\begin{equation}
	u^{n+1}=u^{n}+\theta\tau Mu^{n+1}+(1-\theta)\tau Mu^{n}
	+\theta\tau F(t_{n+1},u^{n+1})+(1-\theta)\tau F(t_{n},u^n)
	+B(t_n,u^n)\Delta W^{n+1},
	\end{equation}
	then via the same procedure as Theorem \ref{thm:order} we could derive the result of mean-square convergence order $1/2$, i.e.,
	\begin{equation}
	\max_{0\leq n\leq N} \left(\mathbb{E}\|e^n\|_{\mathbb H}^2\right)^{1/2}\leq C\tau^{\frac12},
	\end{equation}
	where the positive constant $C$ may depend on the Lipschitz coefficients of $F$ and $B$, $T$, $\|u_0\|_{L^2(\Omega;{\mathcal D}(M^2))}$ and $\|Q^{\frac12}\|_{HS(U,H^{2+\gamma}(D))}$, but independent of $\tau$ and $n$.
	The key character in the proof is the appearance of the positive term $\|e^{n+1}-e^{n}\|^2_{\mathbb H}$ in the right-hand side of \eqref{4.31}, which could absorb the difficulty caused by the stochasticity of the continuous and discrete systems.
\end{remark}
\section{Conclusions}
In this paper, we consider a semi-implicit discretization in temporal direction for stochastic nonlinear Maxwell equations. First, we establish the regularity properties of the continuous and discrete problems. Then based on these regularity properties and utilizing the energy estimate technique, the mean-square convergence order $1/2$ is derived.

Future work will include the study for the full discretization of the stochastic Maxwell equations, in which the error estimates in spatial direction depend on the enough smoothness of the noise covariance and the initial data. Besides, due to the high dimensions and stochasticity of stochastic Maxwell equations, the computational implement is an important and technical issue. In order to approximate this problem efficiently and effectively, some techniques such as splitting approach may be employed, and thus the analysis of the effect on the convergence order induced by these techniques also constitutes future work.

%\ifCLASSOPTIONcompsoc
%  % The Computer Society usually uses the plural form
%  This research was supported b
%\else
%  % regular IEEE prefers the singular form
%\section*{Acknowledgment}
%The first and second authors are supported by the NNSFC (NO. 91130003, NO. 11021101, NO. 11290142, NO. 91630312), the
%third author is supported by the NNSFC (NO. 11601032, NO. 11471310).



%% References


%\bibliographystyle{model1a-num-names}
%\bibliographystyle{siam}
%\bibliography{ref.bib}

\parindent=6mm
\begin{thebibliography}{}
	
	%\bibitem{CHZ2016}
	%C.~Chen, J.~Hong and L.~Zhang, Preservation of physical properties of stochastic Maxwell equations with additive noise via stochastic multi-symplectic methods, J. Comput. Phys., \textbf{306} (2016), 500-519.
	%
	%\bibitem{HJZ2014}
	%J.~Hong, L.~Ji and L.~Zhang, A stochastic multi-symplectic scheme for stochastic Maxwell equations with additive noise, J. Comput. Phys., \textbf{268} (2014), 255-268.
	%
	%\bibitem{HJZC2017}
	%J.~Hong, L.~Ji, L.~Zhang and J.~Cai, An energy-conserving method for stochastic Maxwell equations with multiplicative noise, J. Comput. Phys., \textbf{351} (2017), 216-229.
	%
	%\bibitem{Sko2008}
	%M. I. Skolnik, {\em Radar handbook}, 3rd Edition, McGraw-Hill, New York, 2008.
	%
	%\bibitem{Mos2008}
	%M. Moscoso, Polarization-based optical imaging, in: Inverse Problems and Imaging, 1943 of Lecture Notes in Mathematics, Springer, 2008.
	%
	%\bibitem{AP2005}
	%L.C. Andrews and R.L. Phillips, {\em Laser beam propagation through random media}, SPIE Press, Bellingham, WA, 2005.
	%
	%\bibitem{Gol2005}
	%A. Goldsmith, {\em Wireless Communications}, Cambridge University Press, Cambridge, 2005.	
	%
	%\bibitem{RKT1989}
	%S.M. Rytov, Y.A. Kravtsov and V.I. Tatarskii, {\em Principles of statistical radiophysics 3: Elements of random fields}, Springer, 1989.
	%
	%\bibitem{FM2008}
	%M. Francoeur and M. Meng\"{u}\c{c}, Role of fluctuational electrodynamics in near-field radiative heat transfer, J. Quant. Spectrosc. Radiat. Transf., \textbf{109} (2008), 280-293.
	%
	%\bibitem{LSY2010}
	%K.B. Liaskos, I.G. Stratis and A.N. Yannacopoulos, Stochastic integrodifferential equations in Hilbert spaces with applications in electromagnetics, J. Integral Equations Appl., \textbf{22} (2010), 559-590.
	%
	%\bibitem{Hor2017}
	%L. Hornung, Strong solutions to a nonlinear stochastic Maxwell equations with a retarded material law, arXiv:1703.04461v1 [math.PR] 13 Mar 2017.
	%
	%\bibitem{HSY2010}
	%T. Horsin, I.G. Stratis and A.N. Yannacopoulos, On the approximate controllability of the stochastic Maxwell equations, IMA J. Math. Control. I., \textbf{27} (2010), 103-118.
	%
	%\bibitem{RSY2012}
	%G.F. Roach, I.G. Stratis and A.N. Yannacopoulos, {\em Mathematical analysis of deterministic and stochastic problems in complex media electromagnetics}, Princeton Series in Applied Mathematics, Princeton University Press, Princeton, 2012.
	%
	%\bibitem{Zhang2008}
	%K. Zhang, Numerical studies of some stochastic partial differential equations.
	%Ph.D. Thesis, The Chinese University of Hong Kong, China, 2008.
	%
	%\bibitem{BAZC2010}
	%M. Badieirostami, A. Adibi, H. Zhou and S. Chow, Wiener chaos expansion and simulation of electromagnetic wave propagation excited by a spatially incoherent source, Multiscale Model. Simul., \textbf{8} (2010), 591-604
	%
	%\bibitem{Fann2007}
	%A. Fannjiang, Two-frequency radiative transfer: Maxwell equations in random dielectrics, J. Opt. Soc. Am., A 24, (2007), 3680-3690.
	%
	%\bibitem{Mich2006}
	%B.L. Michielsen, Probabilistic modelling of stochastic interactions between electromagnetic fields and
	%systems, C. R. Phys., \textbf{7} (2006), 543-559.
	%
	%\bibitem{CSJ2006}
	%D.L.C. Chan, M. Soljacic and J.D. Joannopoulos, Direct calculation of thermal emission for threedimensionally
	%periodic photonic crystal slabs, Phys. Rev. E 74, 036615 (2006).
	%
	%\bibitem{Jack1999}
	%J.D. Jackson, {\em Classical Electrodynamics}, 3rd edn. Wiley, New York, 1999.
	%
	%\bibitem{Eyg1972}
	%L. Eyges, {\em The Classical Electromagnetic Fields}, Addison-Wesley, 1972.
	%
	%\bibitem{Liu2015}
	%G. Liu, Stochastic wave propagation in Maxwell's equations, J. Stat. Phys., \textbf{158} (2015), 1126-1146.
	%
	%\bibitem{BSY2015}
	%G. Barbatis, I.G. Stratis and A.N. Yannacopoulos, Homogenization of random elliptic systems with
	%an application to Maxwell's equations, Math. Models Methods Appl. Sci., \textbf{25} (2015), 1365-1388.
	%
	%\bibitem{JWH2013}
	%S. Jiang, L. Wang and J. Hong, Stochastic multi-symplectic integrator for stochastic nonlinear Schr\"{o}dinger equation, Commun. Comput. Phys., \textbf{14} (2013), 393-411
	%
	%\bibitem{HJZ2014}
	%J. Hong, L. Ji and L. Zhang, A stochastic multi-symplectic scheme for stochastic Maxwell equations with additive noise, J. Comput. Phys., \textbf{268} (2014), 255-268
	%
	%\bibitem{CHZ2016}
	%C. Chen, J. Hong and L. Zhang, Preservation of physical properties of stochastic Maxwell equations with additive noise via stochastic multi-symplectic methods, J. Comput. Phys., \textbf{306} (2016), 500-519.
	%
	%\bibitem{HJZC2017}
	%J. Hong, L. Ji, L. Zhang and J. Cai, An energy-conserving method for stochastic Maxwell equations with multiplicative noise, J. Comput. Phys., \text{351} (2017), 216-229
	%
	%\bibitem{PZ2014}
	%G. Da Prato and J. Zabczyk, {\em Stochastic Equations in
	%	Infinite Dimensions}, 2nd edn. Cambridge University Press, United Kingdom, 2014.
	%
	%\bibitem{EN2000}
	%K.-J. Engel and R. Nagel, {\em One-parameter Semigroups for Linear Evolution Equations}, Springer-Verlag, new York, 2000.
	%
	%\bibitem{BT1973}
	%A. Bensoussan and R. Termam, {\em Equations stochastiques du type Navier-Stokes}, J. Funct. Anal., \text{13} (1973), 195-222.
	%
	\bibitem{RKT1989}
	S.M. Rytov, Y.A. Kravtsov and V.I. Tatarskii, {\em Principles of statistical radiophysics 3: Elements of random fields}, Springer, 1989.
	
	\bibitem{FM2008}
	M. Francoeur and M. Meng\"{u}\c{c}, Role of fluctuational electrodynamics in near-field radiative heat transfer, J. Quant. Spectrosc. Radiat. Transf., \textbf{109} (2008), 280-293.
	
%	\bibitem{LSY2010}
%	K.B. Liaskos, I.G. Stratis and A.N. Yannacopoulos, Stochastic integrodifferential equations in Hilbert spaces with applications in electromagnetics, J. Integral Equations Appl., \textbf{22} (2010), 559-590.
	
	\bibitem{RSY2012}
	G. Roach, I.G. Stratis and A.N. Yannacopoulos, Mathematical analysis of deterministic and stochastic problems in complex media electromagnetics, Princeton University Press, 2012.
	
	\bibitem{Zhang2008}
	K. Zhang, Numerical studies of some stochastic partial differential equations.
	Ph.D. Thesis, The Chinese University of Hong Kong, China, 2008.
	
	\bibitem{BAZC2010}
	M. Badieirostami, A. Adibi, H. Zhou and S. Chow, Wiener chaos expansion and simulation of electromagnetic wave propagation excited by a spatially incoherent source, Multiscale Model. Simul., \textbf{8} (2010), 591-604.
	
	\bibitem{Fann2007}
	A. Fannjiang, Two-frequency radiative transfer: Maxwell equations in random dielectrics, J. Opt. Soc. Am., A 24, (2007), 3680-3690.
	
	\bibitem{Mich2006}
	B.L. Michielsen, Probabilistic modelling of stochastic interactions between electromagnetic fields and
	systems, C. R. Phys., \textbf{7} (2006), 543-559.
	
	\bibitem{HJZ2014}
	J. Hong, L. Ji and L. Zhang, A stochastic multi-symplectic scheme for stochastic Maxwell equations with additive noise, J. Comput. Phys., \textbf{268} (2014), 255-268.
	
	\bibitem{CHZ2016}
	C. Chen, J. Hong and L. Zhang, Preservation of physical properties of stochastic Maxwell equations with additive noise via stochastic multi-symplectic methods, J. Comput. Phys., \textbf{306} (2016), 500-519.
	
	\bibitem{HJZC2017}
	J. Hong, L. Ji, L. Zhang and J. Cai, An energy-conserving method for stochastic Maxwell equations with multiplicative noise, J. Comput. Phys., \textbf{351} (2017), 216-229.
	
	\bibitem{PZ2014}
	G. Da Prato and J. Zabczyk, {\em Stochastic Equations in
		Infinite Dimensions}, 2nd edn. Cambridge University Press, United Kingdom, 2014.
	
	\bibitem{EN2000}
	K.-J. Engel and R. Nagel, {\em One-parameter Semigroups for Linear Evolution Equations}, Springer-Verlag, new York, 2000.
	
	\bibitem{BT1973}
	A. Bensoussan and R. Termam, {\em Equations stochastiques du type Navier-Stokes}, J. Funct. Anal., \textbf{13} (1973), 195-222.
	
	\bibitem{Pzur2013}
	T. Pa\v{z}ur, {\em Error analysis of implicit and exponential time integration of linear Maxwell's equations}, PhD Thesis, Karlsruher Institut f\"{u}r Technologie, 2013.
	
	\bibitem{L2015}
	G. Liu, Stochastic wave propagation in Maxwell's equations, J. Stat. Phys.,
	\textbf{158} (2015), 1126-1146.
\end{thebibliography}
%\section*{References}
%\footnotesize

%	\begin{thebibliography}{1}
%		
%		
%		\bibitem{EV}
%		M. Eller and N. P. Valdivia, Acoustic source identification using
%		multiple frequency information, Inverse Problems, 25 (2009), 115005.
%		
%		\bibitem{ZG-IP15}
%		D. Zhang and Y. Guo, Fourier method for solving the multi-frequency inverse
%		acoustic source problem for the Helmholtz equation, Inverse Problems, 31
%		(2015), 035007.
%		
%		
%	\end{thebibliography}
	
\end{document}